\documentclass{siamltex}
\usepackage[T1]{fontenc}
\usepackage[utf8]{inputenc}
\usepackage{lmodern}
\usepackage{amssymb}
\usepackage{amsmath}
\usepackage{amsfonts}
\usepackage{sw20siam}
\usepackage{graphicx}
\usepackage{epsfig}
\usepackage{multicol}
\usepackage{epstopdf}
\usepackage{float}
\usepackage{subfig}

\newtheorem {algorithm}[theorem]{Algorithm}

\newtheorem {remark}[theorem]{Remark}

\input {tcilatex}

\begin{document}

\title{Doubly-Adaptive Artificial Compression Methods for Incompressible Flow}

\author{William Layton \thanks{%
Department of Mathematics, University of Pittsburgh, Pittsburgh, PA 15260,
USA, wjl@pitt.edu; The research herein was partially supported by NSF\
grants DMS1522267, 1817542 and CBET 1609120.} \and Michael McLaughlin 
\thanks{%
Department of Mathematics, University of Pittsburgh, Pittsburgh, PA 15260,
USA, mem266@pitt.edu; The research herein was partially supported by NSF grants DMS1522267, 1817542 and CBET 1609120.}}
\date{\today}
\maketitle

\begin{abstract}
This report presents adaptive artificial compression methods in which the
time-step and artificial compression parameter $\varepsilon $ are
independently adapted. The resulting algorithms are supported by analysis
and numerical tests. The first and second-order methods are embedded. As a
result, the computational, cognitive and space complexities of the adaptive $%
\varepsilon ,k$ algorithms are negligibly greater than that of the simplest,
first-order, constant $\varepsilon ,$ constant $k$ artificial compression
method.
\end{abstract}



\section{Introduction}

\textit{Artificial compression} (AC) methods are based on replacing $\nabla
\cdot u=0$ by $\varepsilon p_{t}+\nabla \cdot u=0$ ($0<\varepsilon $ small),
uncoupling velocity and pressure and advancing the pressure explicitly in
time. Their high speed and low storage requirements recommend them for
complexity bound fluid flow simulations. Unfortunately, \emph{time-accurate}
artificial compression approximations have proven elusive. Time accuracy
(along with increased efficiency and decreased memory) is obtained by \emph{%
time-adaptive} algorithms. To our knowledge, the defect correction based
scheme of Guermond and Minev \cite{GM18} and the non-autonomous AC method in 
\cite{CLM18}, both adapting the time-step with $\varepsilon =k$ (time-step),
are the only previous implicit, time-adaptive AC methods.

This report presents time-adaptive AC algorithms based on a new approach of 
\textit{independently} adapting the AC parameter $\varepsilon $ and
time-step $k$. The methods proceed as follows. A standard, first-order,
implicit method, \eqref{alg:1st} below, is used to advance the \textit{momentum}
equation in the artificial compression equations. A second-order \textit{%
velocity} approximation, \eqref{alg:2nd} below, is then computed at negligible
cost using a time filter adapted from \cite{GL18}. The difference between
the first-order and second-order approximations gives a reliable estimator,
EST(1), for the local error in the momentum equation for the first-order
method and is used to adapt the time step in Algorithm 4.1, Section 4.

Adapting the AC parameter $\varepsilon $\ is more\ challenging. Stability of
the standard AC discrete continuity equation ($\varepsilon p_{t}+\nabla
\cdot u=0$) is unknown for variable $\varepsilon $, \cite{CLM18}. We present
two new, variable $\varepsilon ,$ discrete continuity equations in (1.4)
below\ and prove their unconditional, long-time stability in Theorems 2.1,
2.2 and 3.2. These results show that adaptivity will respond to accuracy
constraints rather than try to correct stability problems with small
time-steps. In these continuity equations, the size of $||\nabla \cdot u||$\
is monitored and used to adapt the choice of the AC parameter $\varepsilon $%
\ (e.g., Algorithm 3.1, Section 3) whereupon the calculation proceeds to the
next time step. The self-adaptive strategy for independently adapting $%
\varepsilon $ also side steps the practical problem of how to pick $%
\varepsilon $ in AC methods and related penalty methods, even for constant
time-steps. The new discrete continuity equations reduce to the standard $%
\varepsilon p_{t}+\nabla \cdot u=0$ for constant $\varepsilon $, improve,
through greater simplicity, a non-autonomous ($\varepsilon =\varepsilon (t)$%
) AC formulation in \cite{CLM18} and yield now three proven stable
extensions of the discrete AC continuity equation to variable $\varepsilon $%
. A comparison of the three is presented in Section 5. Determining if one or
some combination of the three\footnote{%
The stability proof extends to weighted averages of the three discrete
continuity equations.} or some other, yet undetermined, possibility is to be
preferred is an important open problem.

\textbf{The second-order method.} To obtain an $\mathcal{O}(k^{2})$
approximation of the momentum equation (with embedded error estimator),
Algorithms 4.1 and 4.2 incorporate a recent idea of \cite{GL18} of
increasing accuracy and estimating errors by time filters. Theorem 3.2 of
Section 3.1 gives a proof of unconditional, long-time stability of the
second order, constant time-step but variable $\varepsilon $ method. The
resulting embedded structure of Algorithms 4.1 and 4.2 suggests
low-complexity, variable-order methods may be possible once an adaptive $%
\varepsilon $\ strategy is well developed. 

The second-order method is a one leg method. Reliable estimators of the
local truncation error (LTE) in one leg methods are expensive as detailed in 
\cite{DLL18}. An inexpensive estimator, EST(2) in Algorithm 4.2, of the LTE
in the method's linear multistep twin, based on\ a second time filter, is
presented. For the one leg method, this estimator is inexpensive but
heuristic. The doubly adapted, second-order method in Algorithm 4.2 is
tested in Section 5. The embedded structure of the first and second-order
method suggests that adapting the \textit{method order} in addition to the
time-step and AC parameter $\varepsilon $\ may increase accuracy and
efficiency further.

Three stable treatments of the momentum equation (first, second and even
variable order) are possible. Three stable treatments of the variable $%
\varepsilon $\ continuity are now possible: two in (1.1) below and one in 
\cite{CLM18}. The result is nine adaptive AC methods with computational
complexity comparable to the common first-order method, described next.

\subsection{Review of a Common Artificial Compression Method}

Denote by $u$ the velocity, $p$ the pressure, $\nu $ the kinematic
viscosity, and $f$ the external force. Consider the slightly
compressible/hyposonic\footnote{%
We do not include a traditional superscript "$\varepsilon $" as we shall
focus only on AC models and methods.}, \cite{Z06}, approximation to the
incompressible Navier-Stokes equations in a domain $\Omega $ in ${{\mathbb{R}%
}}^{d},d=2,3$ 
\begin{equation}
\begin{cases}
u_{t}+u\cdot \nabla u+{\frac{1}{2}}(\nabla \cdot u)u+\nabla p-\nu \Delta u=f
\\ 
\varepsilon p_{t}+\nabla \cdot u=0,\text{ where }0<\varepsilon \text{ is
small.}%
\end{cases}
\label{AC}
\end{equation}%
This is the most common of several possible formulations reviewed in Section
1.1 of \cite{CLM18}. To present methods herein we will consistently suppress
the secondary spacial discretization\footnote{%
All stability results proven herein hold, by the same proof, for standard
variational spatial discretizations such as finite element methods with
div-stable elements.}. Let $u^{\ast }$ denote the standard (second order)
linear extrapolation of $u$ from previous values\footnote{%
Temperton and Staniforth \cite{TS87} advocated even higher order
extrapolation.} to $t_{n+1}$%
\begin{equation*}
u^{\ast }=\left( 1+\frac{k_{n+1}}{k_{n}}\right) u_{n}-\frac{k_{n+1}}{k_{n}}%
u_{n-1}\left( =2u_{n}-u_{n-1}\text{ for constant time-step}\right) .
\end{equation*}%
To fix ideas, among many possible, e.g., \cite{GMS06}, \cite{GM15}, \cite%
{GM17}, \cite{JL04}, \cite{K02}, \cite{P97}, \cite{DLM17}, \cite{OA10}, \cite%
{YBC16}, consider a common, constant time-step, semi-implicit time
discretization of (\ref{AC}): 
\begin{gather}
\frac{u_{n+1}-u_{n}}{k}+u^{\ast }\cdot \nabla u_{n+1}+{\frac{1}{2}}(\nabla
\cdot u^{\ast })u_{n+1}+\nabla p_{n+1}-\nu \Delta u_{n+1}=f(t_{n+1}),
\label{eq:standardAC1} \\
{\text{ \ \ \ }}\varepsilon \frac{p_{n+1}-p_{n}}{k}+\nabla \cdot u_{n+1}=0. 
\notag
\end{gather}%
Here $k$ is the time-step, $t_{n}=nk,$ $u_{n}$, $p_{n}$ are approximations
to the velocity and pressure at $t=t_{n}$. This has consistency error $%
\mathcal{O}(k+\varepsilon )$ leading to the most common choice of selecting $%
\varepsilon =k$ to balance errors. Since $\nabla p_{n+1}=$ $\nabla
p_{n}-(k/\varepsilon )\nabla \nabla \cdot u_{n+1}$, this uncouples into a
velocity solve followed by an algebraic pressure update%
\begin{gather}
\frac{u_{n+1}-u_{n}}{k}+u^{\ast }\cdot \nabla u_{n+1}+{\frac{1}{2}}(\nabla
\cdot u^{\ast })u_{n+1}-\frac{k}{\varepsilon }\nabla \nabla
\cdot u_{n+1} \notag \\
-\nu \Delta u_{n+1}=-\nabla p_{n}+f_{n+1},  \notag \\
{\text{ {then given }}u_{n+1}\text{{: \ \ \ }}}p_{n+1}=p_{n}-\frac{k}{%
\varepsilon }\nabla \cdot u_{n+1}.  \label{eq:StandardACmethod2}
\end{gather}%
For constant $\varepsilon ,k$, this method is unconditionally, nonlinearly,
long-time stable, e.g., \cite{GMS06}, \cite{GM15}, \cite{S92}, \cite{S92a}.
Its long-time stability for \textit{variable} $\varepsilon ,k$ is an open
problem, \cite{CLM18}.

\subsection{New Methods for Variable $\protect\varepsilon ,k$}

Although well motivated, the \ choice $\varepsilon =k$ cannot be more than a
step to a correct choice. First observe that $Units(\varepsilon
)=Time^{2}/Length^{3}$ while $Units(k)=Time$. Thus, a correct choice of $%
\varepsilon $\ should be scaled to be dimensionally consistent and
afterwards the constant multiplier optimized. Aside from dimensional
inconsistency, the standard choice $\varepsilon =k$ ignores the different
roles of $\varepsilon $ and $k$. To leading orders, the consistency error in
the \textit{continuity} equation is $\mathcal{O}(\varepsilon )$, independent
of $k$, and the consistency error in the \textit{momentum} equation is $%
\mathcal{O}(k)$, independent of $\varepsilon $. This observation on the
standard method (\ref{eq:standardAC1}), (\ref{eq:StandardACmethod2})
motivates the development plan for the doubly adaptive algorithms herein:

\begin{itemize}
\item \textit{Develop first (Section 2) and second (Section 3) order methods
stable for variable }$k,\varepsilon $\textit{.}
\end{itemize}

\begin{itemize}
\item \textit{Adapt }$\varepsilon _{n}$\textit{\ to control the consistency
error in the continuity equation by monitoring }$||\nabla \cdot u||$\textit{%
, Sections 3, 4.}
\end{itemize}

\begin{itemize}
\item \textit{Develop inexpensive estimators for momentum equation
consistency error and adapt }$k=k_{n}$\textit{\ for its control, Section 4.}
\end{itemize}

\begin{itemize}
\item \textit{Use (Section 4) and test (Section 5) the estimators in a
doubly adaptive, variable }$\varepsilon ,k$\textit{, algorithm.}
\end{itemize}

In adaptive methods, \textit{strong stability is necessary}, so $\varepsilon
_{n},k_{n}$ can be adapted for time-accuracy rather than to correct
instabilities. One key difficulty, resolved by the two methods (\ref%
{eqDiscreteContinuityEqns}) below, is that useful \textit{stability is
unknown} for the common AC method (\ref{eq:standardAC1}) with variable $%
\varepsilon $, see \cite{CLM18}, and even for the continuum model (\ref{AC})
with $\varepsilon =\varepsilon (t)$. A second key difficulty is that
(unconditional, nonlinear) G-stability for variable time-steps is uncommon%
\footnote{%
To our knowledge, the only such two-step method is the little explored one
of Dahlquist, Liniger, and Nevanlinna \cite{DLN83}. This second issue may be
resolvable by a variable (first and second) order implementation since it
would include the A-stable, fully implicit method.}. (For example, the
popular BDF2 method loses $A-$stability for increasing time-steps.)

The continuity equation is treated by either a geometric average (GA-Method)
or a minimum term (min-Method) as follows. Given $u_{n},p_{n},\varepsilon
_{n}$, select $\varepsilon _{n+1},k_{n+1}$ calculate $u_{n+1}$\ then%
\footnote{%
A convex combination of the two continuity equations discretizations is also
stable.}%
\begin{equation}
\begin{array}{ccc}
\text{GA-Method:} &  & \frac{\varepsilon _{n+1}p_{n+1}-\sqrt{\varepsilon
_{n+1}\varepsilon _{n}}p_{n}}{k_{n+1}}+\nabla \cdot u_{n+1}=0,\text{ or} \\ 
&  &  \\ 
\text{min-Method:} &  & \frac{\varepsilon _{n+1}p_{n+1}-\min \{\varepsilon
_{n+1},\varepsilon _{n}\}p_{n}}{k_{n+1}}+\nabla \cdot u_{n+1}=0.%
\end{array}
\label{eqDiscreteContinuityEqns}
\end{equation}%
These methods are proven in Section 2 to be unconditionally, variable $%
\varepsilon ,k$ stable. For the discrete momentum equation, recall $u^{\ast }
$ is an extrapolated approximation to $u(t_{n+1})$. The first-order method's
momentum equation is the standard one (\ref{eq:standardAC1}) above given by%
\begin{equation}\label{alg:1st}
\frac{u_{n+1}-u_{n}}{k_{n+1}}+u^{\ast }\cdot \nabla u_{n+1}+{\frac{1}{2}}%
(\nabla \cdot u^{\ast })u_{n+1}+\nabla p_{n+1}-\nu \Delta u_{n+1}=f_{n+1}. 
\tag{1st Order}
\end{equation}%
The (linearly implicit) treatment of the nonlinear term is inspired by Baker 
\cite{B76}. The second method, adapted from \cite{GL18}, adds a time filter
to obtain $\mathcal{O}(k^{2})$\ accuracy and automatic error estimation as
follows. Let the time-step ratio be denoted $\tau =k_{n+1}/k_{n}$. Call $%
u_{n+1}^{1}$\ the solution obtained from the first-order method (1st Order)
above. The second-order approximation $u_{n+1}$ is obtained by filtering $%
u_{n+1}^{1}$:%
\begin{equation}\label{alg:2nd}
\begin{split}
\frac{u_{n+1}^{1}-u_{n}}{k_{n+1}}&+u^{\ast }\cdot \nabla u_{n+1}^{1}+{\frac{1%
}{2}}(\nabla \cdot u^{\ast })u_{n+1}^{1}+\nabla p_{n+1}-\nu \Delta
u_{n+1}^{1}=f_{n+1},  \\
\text{For }\tau &=\frac{k_{n+1}}{k_{n}}\text{ let }\alpha _{1}=\frac{\tau
(1+\tau )}{(1+2\tau )},\text{ then}:  \\
u_{n+1}&=u_{n+1}^{1}-\frac{\alpha _{1}}{2}\left\{ \frac{2k_{n}}{k_{n}+k_{n+1}}%
u_{n+1}^{1}-2u_{n}+\frac{2k_{n+1}}{k_{n}+k_{n+1}}u_{n-1}\right\} .  
\end{split}\tag{2nd
Order}
\end{equation}%
Denote by $D_{2}(n+1)$ the quantity above in braces%
\begin{equation*}
D_{2}(n+1):=\frac{2k_{n}}{k_{n}+k_{n+1}}u_{n+1}^{1}-2u_{n}+\frac{2k_{n+1}}{%
k_{n}+k_{n+1}}u_{n-1}.
\end{equation*}%
Note that $D_{2}(n+1)$ is $2k_{n}k_{n+1}\times $(a second divided
difference).

The usual $L^{2}$ norm $||\cdot ||$ and inner product $(\cdot ,\cdot )$ are
denoted%
\begin{equation*}
||v||=\left( \int_{\Omega }|v(x)|^{2}dx\right) ^{1/2}\text{ and }%
(v,w)=\int_{\Omega }v(x)\cdot w(x)dx.
\end{equation*}%
A simple estimate of the local error in the first-order approximation $%
u_{n+1}^{1}$ is given by a measure (here the\ $L^{2}$ norm) of the
difference of the two approximations%
\begin{equation*}
EST(1)=||u_{n+1}-u_{n+1}^{1}||=\frac{\alpha _{1}}{2}||D_{2}(n+1)||.
\end{equation*}

\textbf{Estimating the error in the second-order approximation.} Naturally
one would like to use the second-order approximation for more\ than an
estimator. It is possible to use $EST(1)$ above as a pessimistic estimator
for $u_{n+1}$. In Section 3 we show that, eliminating the intermediate step $%
u_{n+1}^{1}$, the second-order method is equivalent to the second-order, one
leg method (\ref{eq:secondOrderMoEqn}) below. Estimation of the LTE for this
OLM cannot be done by a simple time filter for reasons delineated in \cite%
{DLL18} and based on classical analysis of the LTE in OLMs of Dahlquist. We
test an inexpensive but heuristic estimator that can be calculated by a
second time filter. $EST(2)$ below is an LTE estimator for the OLMs linear
multi-step twin. To estimate the local error in the second order
approximation we use the third divided difference with multiplier chosen (by
a lengthy but elementary Taylor series calculation) to cancel the first term
of the LTE of the methods linear multi-step\ twin%
\begin{eqnarray*}
EST(2) &=&\frac{\alpha_2}{6}\left\Vert \frac{3k_{n-1}}{%
k_{n+1}+k_{n}+k_{n-1}}D_{2}(n+1)-\frac{3k_{n-1}}{k_{n+1}+k_{n}+k_{n-1}}%
D_{2}(n)\right\Vert\\ \text{ where} \\
\alpha_2 &=&\frac{\tau _{n}(\tau _{n+1}\tau _{n}+\tau _{n}+1)(4\tau
_{n+1}^{3}+5\tau _{n+1}^{2}+\tau _{n+1})}{3(\tau _{n}\tau _{n+1}^{2}+4\tau
_{n}\tau _{n+1}+2\tau _{n+1}+\tau _{n}+1)},\text{ and }\tau
_{n}=k_{n}/k_{n-1}.
\end{eqnarray*}%
The resulting adaptive algorithm uncouples like (\ref{eq:StandardACmethod2})
into a velocity update with a grad-div term then an algebraic pressure
update. More reliable but more expensive estimators are possible. The above
inexpensive but heuristic one is tested herein because the motivation for AC
methods is often based on the need for faster and reduced memory algorithms
in specific applications.

\ Section 2 presents the analysis of the two first-order methods, proving
long-time, unconditional stability for variable $\varepsilon ,k$. This
analysis develops the key treatment of the discrete continuity equation
necessary for stability. Section 3.1 gives a proof of unconditional, long
time stability for the variable $\varepsilon $, constant $k$ \textit{second}
order method. This proof can be extended to decreasing time-steps but not
increasing time-steps.

\subsection{Related work}

Artificial compression (AC) methods were introduced in the 1960's by Chorin,
Oskolkov and Temam. Their mathematical foundation has been extensively
developed by Shen \cite{S96}, \cite{S92a}, \cite{S92}, \cite{S96a} and Prohl 
\cite{P97}. Recent work includes \cite{K02}, \cite{DLM17}, \cite{GM15}, \cite%
{GM17}, \cite{JL04}, \cite{OA10} and \cite{YBC16}. The GA-method (geometric
averaging method) herein is motivated by work in \cite{CHL12}\ for
uncoupling atmosphere-ocean problems stably.

There has been extensive development of adaptive methods for \textit{assured}
accuracy in fully coupled, $v-p$ discretizations, e.g., \cite{HJ07}, and
adaptive methods based on estimates of local truncation errors including 
\cite{HPG15}, \cite{KGGS10}, \cite{VV13}. In complement, the work herein
aims at methods that use less expensive local (rather than global) error
estimators, do not provide assured time-accuracy but \textit{emphasize}
(consistent with the artificial compression methods) low cognitive,
computational, and space complexity. Aside from \cite{CLM18} and Guermond
and Minev \cite{GM18}, extension of implicit, time-adaptive methods to
artificial compression discretizations is undeveloped.

Herein accuracy is increased and local errors estimated by time filters.
Other approaches are clearly possible. Time filters are an important tool in
GFD to correct weak instabilities and extend forecast horizons, \cite{A72}, 
\cite{LLT16}, \cite{R69}, \cite{W09}, \cite{W11}. In \cite{GL18}, it was
noticed that a time filter can also increase the convergence rate of the
backward Euler method and estimate errors. G-stability of the resulting
(constant time-step) time discretization was recently proven for the\textit{%
\ fully-coupled}, velocity-pressure Navier-Stokes equations in \cite{DLZ18}.

\section{First-Order, Variable $k,\protect\varepsilon $ Methods}

This section establishes unconditional, long-time, nonlinear stability of
the two variable $k,\varepsilon $\ first-order methods of Section 1.2 in the
usual $L^{2}(\Omega )$ norm, denoted $||\cdot ||$ with associated inner
product $(\cdot ,\cdot )$. The methods differ in the treatment of the
discrete continuity equation and reduce to the standard AC method (\ref%
{eq:standardAC1}) for constant $\varepsilon ,k$. We prove that the first
order implicit discretization of the momentum equation with both new methods
(\ref{eq:minMethod}), (\ref{eq:GAmethod}) are unconditionally, nonlinearly,
long-time stable without assumptions on $\varepsilon _{n},k_{n}$. We study
these new methods in a bounded, regular domain $\Omega $ subject to the
initial and boundary conditions%
\begin{gather}
u_{0}=u_{0}(x)\text{ and }p_{0}=p_{0}(x),\text{ in }\Omega \text{,}  \notag
\\
u_{n}=0\text{ on }\partial \Omega \text{ \ for }t>0\text{.}  \notag
\end{gather}%
The two, first-order methods are: Given $u_{n},p_{n},\varepsilon _{n}$,
select $\varepsilon _{n+1},k_{n+1}$ and%
\begin{gather}
\frac{u_{n+1}-u_{n}}{k_{n+1}}+u^{\ast }\cdot \nabla u_{n+1}+{\frac{1}{2}}%
(\nabla \cdot u^{\ast })u_{n+1}+\nabla p_{n+1}-\nu \Delta u_{n+1}=f_{n+1}, 
\notag \\
\frac{\varepsilon _{n+1}p_{n+1}-\hat{\varepsilon}p_{n}}{k_{n+1}}+\nabla
\cdot u_{n+1}=0,\text{ where}  \notag \\
\hat{\varepsilon}=\min \{\varepsilon _{n+1},\varepsilon _{n}\}\text{ \
\ for the min-Method and}  \label{eq:minMethod} \\
\hat{\varepsilon}=\sqrt{\varepsilon _{n+1},\varepsilon _{n}}\text{ \ \
\ \ for the GA-Method }  \label{eq:GAmethod}
\end{gather}%
For constant $\varepsilon $ both methods reduce to the standard method (\ref%
{eq:standardAC1}), (\ref{eq:StandardACmethod2}) for which stability is
known. Thus, \textit{the interest is stability for variable} $\varepsilon $.

\textbf{Stability of the min-Method.} It is useful to recall that 
\begin{equation*}
\left( \varepsilon _{n+1}-\varepsilon _{n}\right) ^{+}=\max \{0,\varepsilon
_{n+1}-\varepsilon _{n}\}=\varepsilon _{n+1}-\min \{0,\varepsilon
_{n+1}-\varepsilon _{n}\}.
\end{equation*}

\begin{theorem}[Stability of the min Method]
The variable $\varepsilon ,k$ min-Method is unconditionally, long-time
stable. For any $N>0$ the energy equality holds: 
\begin{gather*}
\frac{1}{2}\int_{\Omega }|u_{N}|^{2}+\varepsilon _{N}|p_{N}|^{2}dx+ \\
\sum_{n=0}^{N-1}\frac{1}{2}\int_{\Omega }\min \{\varepsilon
_{n+1},\varepsilon _{n}\}(p_{n+1}-p_{n})^{2}+\left( \varepsilon
_{n+1}-\varepsilon _{n}\right) ^{+}p_{n+1}{}^{2}\\
+\left( \varepsilon_{n}-\varepsilon _{n+1}\right) ^{+}p_{n}{}^{2}dx 
+\sum_{n=0}^{N-1}\int_{\Omega }\frac{1}{2}|u_{n+1}-u_{n}|^{2}+k_{n+1}\nu
|\nabla u_{n+1}|^{2}dx \\
=\frac{1}{2}\int_{\Omega }|u_{0}|^{2}+\varepsilon
_{0}p_{0}{}^{2}dx+\sum_{n=0}^{N-1}k_{n+1}\int_{\Omega }u_{n+1}\cdot
f_{n+1}dx.
\end{gather*}%
Consequently, the stability bound holds:%
\begin{gather*}
\frac{1}{2}\int_{\Omega }|u_{N}|^{2}+\varepsilon _{N}p_{N}{}^{2}dx+ \\
\sum_{n=0}^{N-1}\frac{1}{2}\int_{\Omega }\min \{\varepsilon
_{n+1},\varepsilon _{n}\}(p_{n+1}-p_{n})^{2}dx+\left( \varepsilon
_{n+1}-\varepsilon _{n}\right) ^{+}p_{n+1}^{2}\\
+\left( \varepsilon_{n}-\varepsilon _{n+1}\right) ^{+}p_{n}{}^{2}dx 
+\sum_{n=0}^{N-1}\frac{1}{2}\int_{\Omega }|u_{n+1}-u_{n}|^{2}+k_{n+1}\nu
|\nabla u_{n+1}|^{2}dx \\
\leq \frac{1}{2}\int_{\Omega }|u_{0}|^{2}+\varepsilon
_{0}p_{0}{}^{2}dx+\sum_{n=0}^{N-1}k_{n+1}\frac{1}{2\nu }||f_{n+1}||_{-1}^{2}.
\end{gather*}
\end{theorem}

\begin{proof}
First we note that using the polarization identity, algebraic rearrangement
and considering the cases $\varepsilon _{n+1}>\varepsilon _{n}$\ and $%
\varepsilon _{n+1}<\varepsilon _{n}$ we have 
\begin{gather*}
(\varepsilon _{n+1}p_{n+1}-\min \{\varepsilon _{n+1},\varepsilon
_{n}\}p_{n},p_{n+1}) \\
=\varepsilon _{n+1}||p_{n+1}||^{2}-\min \{\varepsilon _{n+1},\varepsilon
_{n}\}(p_{n},p_{n+1}) \\
=\varepsilon _{n+1}||p_{n+1}||^{2}-\min \{\varepsilon _{n+1},\varepsilon
_{n}\}\left\{ \frac{1}{2}||p_{n}||^{2}+\frac{1}{2}||p_{n+1}||^{2}-\frac{1}{2}%
||p_{n}-p_{n+1}||^{2}\right\} \\
=\left( \varepsilon _{n+1}-\frac{1}{2}\min \{\varepsilon _{n+1},\varepsilon
_{n}\}\right) ||p_{n+1}||^{2} \\
-\frac{1}{2}\min \{\varepsilon _{n+1},\varepsilon _{n}\}||p_{n}||^{2}+\frac{1%
}{2}\min \{\varepsilon _{n+1},\varepsilon _{n}\}||p_{n}-p_{n+1}||^{2} \\
=\frac{1}{2}\varepsilon _{n+1}||p_{n+1}||^{2}-\frac{1}{2}\varepsilon
_{n}||p_{n}||^{2}+\frac{1}{2}\min \{\varepsilon _{n+1},\varepsilon
_{n}\}||p_{n}-p_{n+1}||^{2}+ \\
+\frac{1}{2}\left( \varepsilon _{n+1}-\min \{\varepsilon _{n+1},\varepsilon
_{n}\}\right) ||p_{n+1}||^{2}+\frac{1}{2}\left( \varepsilon _{n}-\min
\{\varepsilon _{n+1},\varepsilon _{n}\}\right) ||p_{n}||^{2}.
\end{gather*}%
We have $\varepsilon _{n+1}-\min \{\varepsilon _{n+1},\varepsilon
_{n}\}=\left( \varepsilon _{n+1}-\varepsilon _{n}\right) ^{+}$ and $%
\varepsilon _{n}-\min \{\varepsilon _{n+1},\varepsilon _{n}\}=$\newline
$\left( \varepsilon _{n}-\varepsilon _{n+1}\right) ^{+}.$ Thus,%
\begin{gather}
(\varepsilon _{n+1}p_{n+1}-\min \{\varepsilon _{n+1},\varepsilon
_{n}\}p_{n},p_{n+1})=  \label{eq:PressureTermsMinMethod} \\
=\frac{1}{2}\varepsilon _{n+1}||p_{n+1}||^{2}-\frac{1}{2}\varepsilon
_{n}||p_{n}||^{2}+\frac{1}{2}\min \{\varepsilon _{n+1},\varepsilon
_{n}\}||p_{n}-p_{n+1}||^{2}+  \notag \\
+\frac{1}{2}\left( \varepsilon _{n+1}-\varepsilon _{n}\right)
^{+}||p_{n+1}||^{2}+\frac{1}{2}\left( \varepsilon _{n}-\varepsilon
_{n+1}\right) ^{+}||p_{n}||^{2}.  \notag
\end{gather}%
With this identity, take the inner product of the first equation with $%
k_{n+1}u_{n+1}$, the second with $k_{n+1}p_{n+1}$, integrate over the flow
domain, integrate by parts, use skew symmetry, use the polarization identity
twice and add. This yields 
\begin{gather*}
\frac{1}{2}\int_{\Omega
}|u_{n+1}|^{2}-|u_{n}|^{2}+|u_{n+1}-u_{n}|^{2}dx+\int_{\Omega }k_{n+1}\nu
|\nabla u_{n+1}|^{2}dx \\
\frac{1}{2}\int_{\Omega }(\varepsilon _{n+1}p_{n+1}-\min \{\varepsilon
_{n+1},\varepsilon _{n}\}p_{n})p_{n+1}dx=k_{n+1}\int_{\Omega }u_{n+1}\cdot
f_{n+1}dx.
\end{gather*}%
From (\ref{eq:PressureTermsMinMethod}) the energy equality becomes 
\begin{gather*}
\frac{1}{2}\int_{\Omega }|u_{n+1}|^{2}+\varepsilon _{n+1}|p_{n+1}|^{2}dx-%
\frac{1}{2}\int_{\Omega }|u_{n}|^{2}+\varepsilon
_{n}p_{n}{}^{2}dx\\
+\int_{\Omega }k_{n+1}\nu |\nabla u_{n+1}|^{2}dx+ 
\frac{1}{2}\int_{\Omega }(u_{n+1}-u_{n})^{2}+\min \{\varepsilon
_{n+1},\varepsilon _{n}\}(p_{n}-p_{n+1})^{2} \\
+\left( \varepsilon _{n+1}-\varepsilon _{n}\right) ^{+}p_{n+1}{}^{2}+\left(
\varepsilon _{n}-\varepsilon _{n+1}\right)
^{+}p_{n}{}^{2}dx=k_{n+1}\int_{\Omega }u_{n+1}\cdot f_{n+1}dx.
\end{gather*}%
Upon summation the first two terms telescope, completing the proof of the
energy equality. The stability estimate follows from the energy equality and
the Cauchy-Schwarz-Young inequality.
\end{proof}

\bigskip The stability analysis shows that the numerical dissipation in the
min-Method is%
\begin{eqnarray*}
\begin{array}{c}
\text{Numerical} \\ 
\text{Dissipation}%
\end{array}%
\text{ } &\text{=}&\frac{1}{2}k_{n+1}^{2}\int_{\Omega }|\frac{u_{n+1}-u_{n}}{%
k_{n+1}}|^{2}+\min \{\varepsilon _{n+1},\varepsilon _{n}\}(\frac{%
p_{n+1}-p_{n}}{k_{n+1}})^{2}+ \\
&&+\left( \frac{\varepsilon _{n+1}-\varepsilon _{n}}{k_{n+1}}\right)
^{+}p_{n+1}^{2}+\left( \frac{\varepsilon _{n}-\varepsilon _{n+1}}{k_{n+1}}%
\right) ^{+}p_{n}{}^{2}dx.
\end{eqnarray*}

\textbf{The GA-Method. }The proof of stability of the GA-method differs from
the last proof only in the treatment of the variable $\varepsilon $ term,
resulting is a different numerical dissipation for the method.

\begin{theorem}[Stability of GA-Method]
The variable $\varepsilon ,k$, first-order GA-Method is unconditionally,
long-time stable. For any $N>0$ the energy equality holds: 
\begin{gather*}
\frac{1}{2}\int_{\Omega }|u_{N}|^{2}+\varepsilon _{N}|p_{N}|^{2}dx+ \\
+\sum_{n=0}^{N-1}\frac{1}{2}\int_{\Omega }|u_{n+1}-u_{n}|^{2}+(\sqrt{%
\varepsilon _{n+1}}p_{n+1}-\sqrt{\varepsilon _{n}}p_{n})^{2}+2k_{n+1}\nu
|\nabla u_{n+1}|^{2}dx \\
=\frac{1}{2}\int_{\Omega }|u_{0}|^{2}+\varepsilon
_{0}|p_{0}|^{2}dx+\sum_{n=0}^{N-1}k_{n+1}\int_{\Omega }u_{n+1}\cdot f_{n+1}dx
\end{gather*}%
and the stability bound holds:%
\begin{gather*}
\frac{1}{2}\int_{\Omega }|u_{N}|^{2}+\varepsilon _{N}|p_{N}|^{2}dx+ \\
+\sum_{n=0}^{N-1}\left[ \frac{1}{2}\int_{\Omega }|u_{n+1}-u_{n}|^{2}+(\sqrt{%
\varepsilon _{n+1}}p_{n+1}-\sqrt{\varepsilon _{n}}p_{n})^{2}+k_{n+1}\nu
|\nabla u_{n+1}|^{2}dx\right] \\
\leq \frac{1}{2}\int_{\Omega }|u_{0}|^{2}+\varepsilon
_{0}|p_{0}|^{2}dx+\sum_{n=0}^{N-1}k_{n+1}\frac{1}{2\nu }||f_{n+1}||_{-1}^{2}.
\end{gather*}
\end{theorem}

\begin{proof}
First we note that using the polarization identity we have 
\begin{gather*}
(\varepsilon _{n+1}p_{n+1}-\sqrt{\varepsilon _{n+1}\varepsilon _{n}}%
p_{n},p_{n+1})= \\
=\varepsilon _{n+1}||p_{n+1}||^{2}-(\sqrt{\varepsilon _{n}}p_{n},\sqrt{%
\varepsilon _{n+1}}p_{n+1}) \\
=\varepsilon _{n+1}||p_{n+1}||^{2}-\left\{ \frac{1}{2}\varepsilon
_{n}||p_{n}||^{2}+\frac{1}{2}\varepsilon _{n+1}||p_{n+1}||^{2}-\frac{1}{2}||%
\sqrt{\varepsilon _{n}}p_{n}-\sqrt{\varepsilon _{n+1}}p_{n+1}||^{2}\right\}
\\
=\frac{1}{2}\varepsilon _{n+1}||p_{n+1}||^{2}-\frac{1}{2}\varepsilon
_{n}||p_{n}||^{2}+\frac{1}{2}||\sqrt{ \varepsilon _{n+1}}p_{n+1}-\sqrt{%
\varepsilon _{n}}p_{n}||^{2}.
\end{gather*}%
The remainder of the proof is the same as for the min-Method.
\end{proof}

The stability analysis shows that the numerical dissipation in the GA-Method
is%
\begin{equation*}
\begin{array}{c}
\text{Numerical} \\ 
\text{Dissipation}%
\end{array}%
\text{ = }\frac{1}{2}k_{n+1}^{2}\int_{\Omega }\left[ \left\vert \frac{%
u_{n+1}-u_{n}}{k_{n+1}}\right\vert ^{2}+\left( \frac{\sqrt{\varepsilon _{n+1}%
}p_{n+1}-\sqrt{\varepsilon _{n}}p_{n}}{k_{n+1}}\right) ^{2}\right] dx.
\end{equation*}%
There is no obvious way to tell \emph{\'{a} priori} which method's numerical
dissipation is larger or to be preferred. A numerical comparison is thus
presented in Section 5.

\begin{remark}
\textbf{The continuum analogs. }It is natural to ask if there is a
non-autonomous continuum AC model associated with each method. The momentum
equation for each continuum model is the standard 
\begin{equation*}
u_{t}+u\cdot \nabla u+{\frac{1}{2}}(\nabla \cdot u)u+\nabla p-\nu \Delta u=f.
\end{equation*}%
The associated continuum continuity equation for the min-Method is 
\begin{equation}
\varepsilon (t)p_{t}+\varepsilon _{t}^{+}p+\nabla \cdot u=0,
\end{equation}%
whereas the continuum continuity equation for the GA-method is 
\begin{equation*}
\sqrt{\varepsilon }(\sqrt{\varepsilon }p)_{t}+\nabla \cdot {u}=0.
\end{equation*}%
Analyzing convergence of each to a weak solution of the incompressible NSE
as (non-autonomous) $\varepsilon (t)\rightarrow 0$\ is a significant open
problem.
\end{remark}

\section{Second-Order, Variable $\protect\varepsilon $ Methods}

The first-order methods are now extended to embedded first and second-order
methods adapting \cite{GL18} from ODEs to the NSE. First we review the idea
of extension used.

\textbf{Review of the ODE algorithm.} Consider the initial value problem%
\begin{equation*}
y^{\prime }(t)=f(t,y(t)),y(0)=y_{0}.
\end{equation*}%
Recall $\tau =k_{n+1}/k_{n}$ is the time-step ratio. The second-order
accurate, variable time-step method of \cite{GL18} is the standard backward
Euler (fully implicit) method followed by a time filter:%
\begin{equation}
\begin{array}{cc}
\text{Step 1} & \frac{y_{n+1}^{1}-y_{n}}{k_{n+1}}=f(t_{n+1},y_{n+1}^{1}), \\ 
& \text{ pick filter parameter }\alpha (1)=\frac{\tau (1+\tau )}{(1+2\tau )}%
\text{, then} \\ 
\text{Step 2} & y_{n+1}=y_{n+1}^{1}-\frac{\alpha_1}{2}\left\{ \frac{2k_{n}%
}{k_{n}+k_{n+1}}y_{n+1}^{1}-2y_{n}+\frac{2k_{n+1}}{k_{n}+k_{n+1}}%
y_{n-1}\right\} .%
\end{array}
\label{eq:VariableMethod}
\end{equation}%
The combination is second-order accurate, $A-$stable for constant or
decreasing time-steps and a measure of the pre- and post-filter difference 
\begin{equation}
EST(1)=|y_{n+1}^{1}-y_{n+1}|
\end{equation}%
can be used in a standard way as a local error estimator for the lower order
approximation $y_{n+1}^{1}$ or a (pessimistic) estimator for the higher
order approximation $y_{n+1}$.

\textbf{A simple, adaptive}$-\varepsilon $\textbf{, second-order AC algorithm%
}. The continuity equation for both methods can be written%
\begin{equation*}
\frac{\varepsilon _{n+1}p_{n+1}-\hat{\varepsilon}p_{n}}{k_{n+1}}+\nabla
\cdot u_{n+1}=0\text{ where }\hat{\varepsilon}=\sqrt{\varepsilon
_{n+1}\varepsilon _{n}}\text{ or }\min \{\varepsilon _{n+1},\varepsilon
_{n}\}\text{.}
\end{equation*}%
This can be used to uncouple velocity and pressure using%
\begin{equation*}
\nabla p_{n+1}=\frac{\hat{\varepsilon}}{\varepsilon _{n+1}}\nabla p_{n}-%
\frac{k_{n+1}}{\varepsilon _{n+1}}\nabla \nabla \cdot u_{n+1}.
\end{equation*}%
The discrete momentum equation for either first-order method is then%
\begin{gather*}
\frac{u_{n+1}^{1}-u_{n}}{k_{n+1}}+u^{\ast }\cdot \nabla u_{n+1}^{1}+{\frac{1%
}{2}}(\nabla \cdot u^{\ast })u_{n+1}^{1}-\frac{k_{n+1}}{\varepsilon _{n+1}}%
\nabla \nabla \cdot u_{n+1}^{1}\\
-\nu \Delta u_{n+1}^{1}=f_{n+1}-\frac{%
\hat{\varepsilon}}{\varepsilon _{n+1}}\nabla p_{n}.
\end{gather*}%
Applying the time filter of (\ref{eq:VariableMethod}) to the velocity
approximation increases the methods accuracy to $\mathcal{O}(k^{2})$. This
combination yields a simple, second-order, constant time-step but adaptive $%
\varepsilon $ algorithm. In the algorithm below the change in $\varepsilon $
is restricted to be between halving and doubling the previous $\varepsilon $%
\ value.

\begin{algorithm}
\textbf{[Simple, adaptive }$\varepsilon $\textbf{, constant time-step,
second-order AC method]. }\texttt{Given} $u_{n},u_{n-1},p_{n},k,\varepsilon
_{n+1},\varepsilon _{n}$, \texttt{and tolerance }$TOL_c$\texttt{, }

\texttt{Select: \ }$\hat{\varepsilon}=\sqrt{\varepsilon
_{n+1}\varepsilon _{n}}$ $\ \mathtt{or}$ $\hat{\varepsilon}=\min
\{\varepsilon _{n+1},\varepsilon _{n}\}$\texttt{\ \ }

\ \texttt{Set: \ }$u^{\ast }=2u_{n}-u_{n-1}.$\texttt{\ \ }

\texttt{Solve for} $u_{n+1}^{1}$%
\begin{gather*}
\frac{u_{n+1}^{1}-u_{n}}{k}+u^{\ast }\cdot \nabla u_{n+1}^{1}+{\frac{1}{2}}%
(\nabla \cdot u^{\ast })u_{n+1}^{1}-\frac{k}{\varepsilon _{n+1}}\nabla
\nabla \cdot u_{n+1}^{1}\\
-\nu \Delta u_{n+1}^{1}=f_{n+1}-\frac{\hat{%
\varepsilon}}{\varepsilon _{n+1}}\nabla p_{n}.
\end{gather*}

\texttt{Filter, Compute estimator} $EST_{c}$ , \texttt{Find} $p_{n+1}$ 
\begin{eqnarray*}
u_{n+1} &=&u_{n+1}^{1}-\frac{1}{3}\left\{ u_{n+1}^{1}-2u_{n}+u_{n-1}\right\}
, \\
EST_{c} &=&||\nabla \cdot u_{n+1}||=\frac{1}{3}\left\Vert
u_{n+1}^{1}-2u_{n}+u_{n-1}\right\Vert , \\
p_{n+1} &=&\frac{\hat{\varepsilon}}{\varepsilon _{n+1}}p_{n}-\frac{%
k_{n+1}}{\varepsilon _{n+1}}\nabla \cdot u_{n+1}.
\end{eqnarray*}

\texttt{Adapt} $\varepsilon :$ \ \texttt{IF} $EST_{c}>TOL_c$ , \texttt{THEN
repeat step after resetting} $\varepsilon _{n+1}$\ \texttt{by} 
\begin{equation*}
\varepsilon _{n+1}=\max \{0.9\varepsilon _{n+1}\frac{TOL_c}{EST_{c}}%
,0.5\varepsilon _{n+1}\}
\end{equation*}

\texttt{ELSE} \ 
\begin{equation*}
\varepsilon _{n+2}=\max \{\min \{0.9\varepsilon _{n+1}\frac{TOL_c}{EST_{c}}%
,2\varepsilon _{n+1}\},.5\varepsilon _{n+1}\}
\end{equation*}

\texttt{and proceed to next step.}
\end{algorithm}

\subsection{Stability of the second-order method for variable $\protect%
\varepsilon $, constant $k$}

This section establishes unconditional, nonlinear, long-time stability of
the second-order GA-method for constant time-steps but variable $\varepsilon 
$. The proof addresses the interaction between the filter step with the
continuity equation. It is adapted to the min-Method following ideas in the
proof of Theorem 2.1. For constant time-steps and variable $\varepsilon $
the GA-method is as follows. Given $u_{n},p_{n},\varepsilon _{n}$, select $%
\varepsilon _{n+1}$ and $u^{\ast }=$ $2u_{n}-u_{n-1}$ (since the time-step
is here constant). Then, 
\begin{gather}
\frac{u_{n+1}^{1}-u_{n}}{k}+u^{\ast }\cdot \nabla u_{n+1}^{1}+{\frac{1}{2}}%
(\nabla \cdot u^{\ast })u_{n+1}^{1}+\nabla p_{n+1}-\nu \Delta
u_{n+1}^{1}=f_{n+1},  \notag \\
\text{Filter: }u_{n+1}=u_{n+1}^{1}-\frac{1}{3}\left\{
u_{n+1}^{1}-2u_{n}+u_{n-1}\right\}  \label{eq:SecOrdConstTimeStepVarEpsilon}
\\
\text{Find }p_{n+1}\text{ :\ }\frac{\varepsilon _{n+1}p_{n+1}-\sqrt{%
\varepsilon _{n+1}\varepsilon _{n}}p_{n}}{k}+\nabla \cdot u_{n+1}^{1}=0\text{
\& proceed to next step.}  \notag
\end{gather}%
We now prove an energy equality for the method which implies stability.

\begin{theorem}
The method (\ref{eq:SecOrdConstTimeStepVarEpsilon}) satisfies the following
discrete energy equality (from which stability follows). For any $N>1$%
\begin{gather*}
\left[ \frac{1}{4}\int_{\Omega
}|u_{N+1}|^{2}+|2u_{N+1}-u_{N}|^{2}+|u_{N+1}-u_{N}|^{2}+2\varepsilon
_{N+1}|p_{N+1}|^{2}dx\right] \\
+\sum_{n=1}^{N}\int_{\Omega }\frac{3}{4}|u_{n+1}-2u_{n}+u_{n-1}|^{2}+\frac{1%
}{2}|\sqrt{\varepsilon _{n+1}}p_{n+1}-\sqrt{\varepsilon _{n}}p_{n}|^{2}dx+ \\
+\sum_{n=1}^{N}k\int_{\Omega }\nu |{\normalsize \nabla }\left[ \frac{3}{2}%
u_{n+1}-u_{n}+\frac{1}{2}u_{n-1}\right] |^{2}dx+ \\
=\left[ \frac{1}{4}\int_{\Omega
}|u_{1}|^{2}+|2u_{1}-u_{0}|^{2}+|u_{1}-u_{0}|^{2}+2\varepsilon
_{1}|p_{1}|^{2}\right] \\
+k\sum_{n=1}^{N}\int_{\Omega }f_{n+1}\cdot \left( \frac{3}{2}u_{n+1}-u_{n}+%
\frac{1}{2}u_{n-1}\right) dx.
\end{gather*}
\end{theorem}

\begin{proof}
To prove stability, eliminate the intermediate value $u_{n+1}^{1}$\ in the
momentum equation. From the filter step $u_{n+1}=u_{n+1}^{1}-\frac{1}{3}%
\left\{ u_{n+1}^{1}-2u_{n}+u_{n-1}\right\} $\ we have%
\begin{equation*}
u_{n+1}^{1}=\frac{3}{2}u_{n+1}-u_{n}+\frac{1}{2}u_{n-1}.
\end{equation*}%
Replacing $u_{n+1}^{1}$ by $\frac{3}{2}u_{n+1}-u_{n}+\frac{1}{2}u_{n-1}$
yields the equivalent discrete momentum equation:%
\begin{gather}
\frac{\frac{3}{2}u_{n+1}-2u_{n}+\frac{1}{2}u_{n-1}}{k}+  \notag \\
+u_{n}^{\ast }\cdot \nabla \left( \frac{3}{2}u_{n+1}-u_{n}+\frac{1}{2}%
u_{n-1}\right) +{\frac{1}{2}}(\nabla \cdot u_{n}^{\ast })\left( \frac{3}{2}%
u_{n+1}-u_{n}+\frac{1}{2}u_{n-1}\right)  \label{eq:secondOrderMoEqn} \\
+\nabla p_{n+1}-\nu \Delta \left( \frac{3}{2}u_{n+1}-u_{n}+\frac{1}{2}%
u_{n-1}\right) =f_{n+1}.  \notag
\end{gather}%
Multiply by the time-step $k$, take the $L^{2}$ inner product of the
momentum equation (\ref{eq:secondOrderMoEqn}) with $\frac{3}{2}u_{n+1}-u_{n}+%
\frac{1}{2}u_{n-1}$, the $L^{2}$ inner product of the discrete continuity
equation with $p_{n+1}$ and add. Two pressure terms cancel since $%
u_{n+1}^{1}=\frac{3}{2}u_{n+1}-u_{n}+\frac{1}{2}u_{n-1}$and the nonlinear
terms vanish due to skew-symmetry. Thus, we obtain%
\begin{gather*}
\left( \frac{3}{2}u_{n+1}-2u_{n}+\frac{1}{2}u_{n-1},\frac{3}{2}u_{n+1}-u_{n}+%
\frac{1}{2}u_{n-1}\right) + \\
+\left( \varepsilon _{n+1}p_{n+1}-\sqrt{\varepsilon _{n+1}\varepsilon _{n}}%
p_{n},p_{n+1}\right) \\
+\nu k\left\Vert \nabla \left[ \frac{3}{2}u_{n+1}-u_{n}+\frac{1}{2}u_{n-1}%
\right] \right\Vert ^{2}=k\left( f_{n+1},\frac{3}{2}u_{n+1}-u_{n}+\frac{1}{2}%
u_{n-1}\right)
\end{gather*}%
The key terms are the first two. For the first term, apply the following
identity from \cite{DLZ18} with $a=u_{n+1},$ $b=u_{n},$ $c=u_{n-1}$%
\begin{gather*}
\left[ \frac{a^{2}}{4}+\frac{(2a-b)^{2}}{4}+\frac{\left( a-b\right) ^{2}}{4}%
\right] -\left[ \frac{b^{2}}{4}+\frac{(2b-c)^{2}}{4}+\frac{\left( b-c\right)
^{2}}{4}\right] \\
+\frac{3}{4}(a-2b+c)^{2}=(\frac{3}{2}a-2b+\frac{1}{2}c)(\frac{3}{2}a-b+\frac{%
1}{2}c).
\end{gather*}%
This yields%
\begin{gather*}
\left( \frac{3}{2}u_{n+1}-2u_{n}+\frac{1}{2}u_{n-1},\frac{3}{2}u_{n+1}-u_{n}+%
\frac{1}{2}u_{n-1}\right) = \\
\left[ \frac{1}{4}||u_{n+1}||^{2}+\frac{1}{4}||2u_{n+1}-u_{n}||^{2}+\frac{1}{%
4}||u_{n+1}-u_{n}||^{2}\right] \\
-\left[ \frac{1}{4}||u_{n}||^{2}+\frac{1}{4}||2u_{n}-u_{n-1}||^{2}+\frac{1}{4%
}||u_{n}-u_{n-1}||^{2}\right] \\
+\frac{3}{4}||u_{n+1}-2u_{n}+u_{n-1}||^{2}.
\end{gather*}%
For the pressure term $\left( \sqrt{\varepsilon _{n+1}\varepsilon _{n}}%
p_{n},p_{n+1}\right) $ the polarization identity, suitably applied, yields%
\begin{gather*}
\left( \sqrt{\varepsilon _{n+1}\varepsilon _{n}}p_{n},p_{n+1}\right) =\left( 
\sqrt{\varepsilon _{n}}p_{n},\sqrt{\varepsilon _{n+1}}p_{n+1}\right) = \\
=\frac{1}{2}\varepsilon _{n+1}||p_{n+1}||^{2}+\frac{1}{2}\varepsilon
_{n}||p_{n}||^{2}-\frac{1}{2}||\sqrt{\varepsilon _{n+1}}p_{n+1}-\sqrt{%
\varepsilon _{n}}p_{n}||^{2}.
\end{gather*}%
Thus%
\begin{gather*}
\left( \varepsilon _{n+1}p_{n+1}-\sqrt{\varepsilon _{n+1}\varepsilon _{n}}%
p_{n},p_{n+1}\right) = \\
=\frac{1}{2}\varepsilon _{n+1}||p_{n+1}||^{2}-\frac{1}{2}\varepsilon
_{n}||p_{n}||^{2}+\frac{1}{2}||\sqrt{\varepsilon _{n+1}}p_{n+1}-\sqrt{%
\varepsilon _{n}}p_{n}||^{2}.
\end{gather*}%
Combining the pressure and velocity identities, we have%
\begin{gather*}
\left[ \frac{1}{4}||u_{n+1}||^{2}+\frac{1}{4}||2u_{n+1}-u_{n}||^{2}+\frac{1}{%
4}||u_{n+1}-u_{n}||^{2}+\frac{\varepsilon _{n+1}}{2}||p_{n+1}||^{2}\right] \\
-\left[ \frac{1}{4}||u_{n}||^{2}+\frac{1}{4}||2u_{n}-u_{n-1}||^{2}+\frac{1}{4%
}||u_{n}-u_{n-1}||^{2}+\frac{\varepsilon _{n}}{2}||p_{n}||^{2}\right] + \\
+\frac{3}{4}||u_{n+1}-2u_{n}+u_{n-1}||^{2}+\frac{1}{2}||\sqrt{\varepsilon
_{n+1}}p_{n+1}-\sqrt{\varepsilon _{n}}p_{n}||^{2} \\
+\nu k\left\Vert {\normalsize \nabla }\left[ \frac{3}{2}u_{n+1}-u_{n}+\frac{1%
}{2}u_{n-1}\right] \right\Vert ^{2}=k\left( f_{n+1},\frac{3}{2}u_{n+1}-u_{n}+%
\frac{1}{2}u_{n-1}\right) .
\end{gather*}%
Summing from $n=1$ to $N$ proves unconditional, long-time stability.
\end{proof}

\section{Doubly $k,\protect\varepsilon $ Adaptive Algorithms}

We present three doubly adaptive AC algorithms: \textit{first-order}, 
\textit{second-order method} and a third that \textit{adapts the method order%
}. The first two are tested in Section 5. While not tested herein, we
include the variable order adaptive algorithm for its clear interest. In the
first algorithm, the error is estimated by a time filter and the next
time-step and next $\varepsilon $ are adapted\footnote{%
The formula for $\varepsilon _{new}$ could be improvable.} based on 
\begin{equation*}
\text{first-order prediction}\text{: }k_{new}=k_{old}\left( \frac{TOL_m}{EST(1)%
}\right) ^{1/2}\text{ and \ }\varepsilon _{new}=\varepsilon _{old}\frac{TOL_c}{%
||\nabla \cdot u_{n+1}||}.
\end{equation*}%
In our implementation, a safety factor of $0.9$ is used and the maximum
change in both is (additionally) restricted to be between $0.5$ \& $2.0$.

\begin{algorithm}[Doubly $k$, $\protect\varepsilon $ Adaptive, First-Order
Method]
\texttt{Given }$TOL_m$\texttt{,}\newline
$TOL_{c}$, $u_{n}$, $u_{n-1}$, $u_{n-2}$ \texttt{and }$k_{n+1}$, $k_{n}$, $%
k_{n-1}$\texttt{\ }

\texttt{Compute: }$\tau =\frac{k_{n+1}}{k_{n}}$\texttt{\ and} \ $\alpha_1=%
\frac{{\LARGE \tau (1.0+\tau )}}{{\LARGE 1.0+2.0\tau }}$

\texttt{Select: \ }$\hat{\varepsilon}=\sqrt{\varepsilon
_{n+1}\varepsilon _{n}}$ $\mathtt{or}$ $\hat{\varepsilon}=\min
\{\varepsilon _{n+1},\varepsilon _{n}\}.$\texttt{\ \ }\ 

\texttt{Set \ \ }$u^{\ast }=\left( 1+\tau \right) u_{n}-\tau u_{n-1}.$

\texttt{Find BE approximation} $u_{n+1}$%
\begin{equation*}
\frac{u_{n+1}-u_{n}}{k_{n+1}}+u^{\ast }\cdot \nabla u_{n+1}+{\frac{1}{2}}%
(\nabla \cdot u^{\ast })u_{n+1}-\frac{k_{n+1}}{\varepsilon _{n+1}}\nabla
\nabla \cdot u_{n+1}-\nu \Delta u_{n+1}=f_{n+1}-\frac{\hat{\varepsilon}%
}{\varepsilon _{n+1}}\nabla p_{n}.
\end{equation*}

\texttt{Compute difference }$D_{2}$\texttt{\ and Estimators}%
\begin{gather*}
D_{2}=\frac{2k_{n}}{k_{n}+k_{n+1}}u_{n+1}^{1}-2u_{n}+\frac{2k_{n+1}}{%
k_{n}+k_{n+1}}u_{n-1} \\
EST(1)=\frac{\alpha_1}{2}||D_{2}||, \\
ESTc=||\nabla \cdot u_{n+1}||.
\end{gather*}

\texttt{IF} $EST_{c}>TOL_{c}$ \ \texttt{or} $EST(1)>TOL_{m}$ \texttt{THEN
repeat step after resetting} $\varepsilon _{n+1},k_{n+1}$\ \texttt{by} 
\begin{eqnarray*}
\varepsilon _{n+1} &=&\max \{0.9\varepsilon _{n+1}\frac{TOL_c}{EST_{c}}%
,0.5\varepsilon _{n+1}\} \\
k_{n+1} &=&0.9\ast \left( \frac{TOL_m}{EST(1)}\right) ^{1/2}\ \ \ \max \left\{
0.9k_{n}\left( \frac{TOL_{m}}{EST(1)}\right) ^{1/2},0.5k_{n+1}\right\}
\end{eqnarray*}

\texttt{ELSE Predict best next step for each approximation:}%
\begin{eqnarray*}
k_{n+2} &=&\max \left\{ \min \left\{ 0.9k_{n+1}\left( \frac{TOL_{m}}{EST(1)}%
\right) ^{1/2},2k_{n+1}\right\} ,0.5k_{n+1}\right\} \\
\varepsilon _{n+2} &=&\max \{\min \{0.9\varepsilon _{n+1}\frac{TOL_c}{EST_{c}}%
,2\varepsilon _{n+1}\},0.5\varepsilon _{n+1}\}
\end{eqnarray*}%
\ 

\ \ \texttt{ENDIF}

\texttt{Update pressure: \ \ }$p_{n+1}=\frac{\hat{\varepsilon}}{%
\varepsilon _{n+1}}p_{n}-\frac{k_{n+1}}{\varepsilon _{n+1}}\nabla \cdot
u_{n+1}.$

\texttt{Proceed to next step.}
\end{algorithm}

\textbf{The second-order, doubly adaptive algorithm.} For the second-order,
doubly adaptive method, we predict the next $\varepsilon $ value the same as
in the first-order method and predict the next time step based on 
\begin{equation*}
\text{second-order\ prediction}\text{: }k_{new}=k_{old}\left( \frac{TOL_m}{%
EST(2)}\right) ^{1/3}.
\end{equation*}%
$EST(2)$ is calculated as follows. The second-order method is equivalent,
after elimination of the intermediate (first-order) approximation, to a one
leg method exactly as in (\ref{eq:SecOrdConstTimeStepVarEpsilon}) in the
constant time-step case. The one leg method's linear multistep twin has
local error proportionate to $k^{3}u_{ttt}+O(k^{4})$. Thus, an estimate of $%
u_{ttt}$\ is computed using difference of $D_{2}$ as follows. Write 
\begin{equation*}
D_{2}(n+1)=\frac{2k_{n}}{k_{n}+k_{n+1}}u_{n+1}^{1}-2u_{n}+\frac{2k_{n+1}}{%
k_{n}+k_{n+1}}u_{n-1}\text{ }
\end{equation*}%
From differences of $D_{2}(n+1)$, $D_{2}(n)$ we obtain the estimator:%
\begin{equation*}
EST(2)=\frac{\alpha_2}{6}\left\Vert \frac{3k_{n-1}}{k_{n+1}+k_{n}+k_{n-1}}%
D_{2}(n+1)-\frac{3k_{n-1}}{k_{n+1}+k_{n}+k_{n-1}}D_{2}(n)\right\Vert ,
\end{equation*}%
where the coefficient $\alpha_2$ is determined through a Taylor series
calculation to be 
\begin{equation*}
\alpha_2=\frac{\tau _{n}(\tau _{n+1}\tau _{n}+\tau _{n}+1)(4\tau
_{n+1}^{3}+5\tau _{n+1}^{2}+\tau _{n+1})}{3(\tau _{n}\tau _{n+1}^{2}+4\tau
_{n}\tau _{n+1}+2\tau _{n+1}+\tau _{n}+1)}
\end{equation*}%
\ 

\begin{algorithm}[Doubly Adaptive, Second-Order Algorithm]
\texttt{Given }$TOL_m$\texttt{, }$TOL_{c}$\texttt{,} $u_{n}$\texttt{,} $u_{n-1}
$\texttt{,} $u_{n-2}$, \texttt{previous 2nd difference} $D_{2}(n)$ \texttt{%
and }$k_{n+1}$\texttt{,} $k_{n}$\texttt{,} $k_{n-1}$\texttt{\ }

\texttt{Compute: }$\tau =\frac{k_{n+1}}{k_{n}},$ $\alpha_1=\frac{{\LARGE %
\tau (1.0+\tau )}}{{\LARGE 1.0+2.0\tau }}$,\ $\alpha_2=\frac{\tau
_{n}(\tau _{n+1}\tau _{n}+\tau _{n}+1)(4\tau _{n+1}^{3}+5\tau
_{n+1}^{2}+\tau _{n+1})}{3(\tau _{n}\tau _{n+1}^{2}+4\tau _{n}\tau
_{n+1}+2\tau _{n+1}+\tau _{n}+1)}$

\texttt{Select: \ \ }$\hat{\varepsilon}=\sqrt{\varepsilon
_{n+1}\varepsilon _{n}}$ $\mathtt{or}$ $\hat{\varepsilon}=\min
\{\varepsilon _{n+1},\varepsilon _{n}\}.$\ 

\texttt{Set: \ \ }$u^{\ast }=\left( 1+\tau \right) u_{n}-\tau u_{n-1}.$

\texttt{Find BE approximation} $u_{n+1}^{1}$%
\begin{equation*}
\frac{u_{n+1}^{1}-u_{n}}{k_{n+1}}+u^{\ast }\cdot \nabla u_{n+1}^{1}+{\frac{1%
}{2}}(\nabla \cdot u^{\ast })u_{n+1}^{1}-\frac{k_{n+1}}{\varepsilon _{n+1}}%
\nabla \nabla \cdot u_{n+1}^{1}-\nu \Delta u_{n+1}^{1}=f_{n+1}-\frac{%
\hat{\varepsilon}}{\varepsilon _{n+1}}\nabla p_{n}.
\end{equation*}

\texttt{Compute difference }$D_{2}$\texttt{\ and update velocity}%
\begin{eqnarray*}
D_{2}(n+1) &=&\frac{2k_{n}}{k_{n}+k_{n+1}}u_{n+1}^{1}-2u_{n}+\frac{2k_{n+1}}{%
k_{n}+k_{n+1}}u_{n-1} \\
u_{n+1} &=&u_{n+1}^{1}-\frac{\alpha_1}{2}D_{2}(n+1)
\end{eqnarray*}

\texttt{Compute estimators}%
\begin{eqnarray*}
EST(2) &=&\frac{\alpha_2}{6}\left\Vert \frac{3k_{n-1}}{%
k_{n+1}+k_{n}+k_{n-1}}D_{2}(n+1)-\frac{3k_{n-1}}{k_{n+1}+k_{n}+k_{n-1}}%
D_{2}(n)\right\Vert , \\
ESTc &=&||\nabla \cdot u_{n+1}||.
\end{eqnarray*}

\texttt{IF} $EST_{c}>TOL_{c}$ \ \texttt{or} $EST(2)>TOL_{m}$ \texttt{THEN
repeat step after resetting} $\varepsilon _{n+1},k_{n+1}$\ \texttt{by} 
\begin{eqnarray*}
\varepsilon _{n+1} &=&\max \{0.9\varepsilon _{n+1}\frac{TOL_c}{EST_{c}}%
,0.5\varepsilon _{n+1}\} \\
k_{n+1} &=&\max \left\{ \min \left\{ 0.9k_{n+1}\left( \frac{TOL_{m}}{EST(2)}%
\right) ^{1/3},2k_{n+1}\right\} ,0.5k_{n+1}\right\}
\end{eqnarray*}

\texttt{ELSE Predict best next step:}%
\begin{eqnarray*}
k_{n+2} &=&\max \left\{ \min \left\{ 0.9k_{n+1}\left( \frac{TOL_{m}}{EST(2)}%
\right) ^{1/3},2k_{n+1}\right\} ,0.5k_{n+1}\right\} \\
\varepsilon _{n+2} &=&\max \{\min \{0.9\varepsilon _{n+1}\frac{TOL_c}{EST_{c}}%
,2\varepsilon _{n+1}\},0.5\varepsilon _{n+1}\}
\end{eqnarray*}%
\ 

\texttt{Update pressure: \ \ }$p_{n+1}=\frac{\hat{\varepsilon}}{%
\varepsilon _{n+1}}p_{n}-\frac{k_{n+1}}{\varepsilon _{n+1}}\nabla \cdot
u_{n+1}.$

\texttt{Proceed to next step.}
\end{algorithm}

\textbf{The adaptive order, time-step and }$\varepsilon $\textbf{\ algorithm.%
} To adapt $\varepsilon ,k$ and the method order we use the local truncation
error indicators for the momentum and continuity equations, respectively, 
\begin{equation*}
\begin{array}{ccc}
\text{Adapt }k\text{ for }u^{1}\text{ using} & : & EST(1) \\ 
\text{Adapt }k\text{ for }u\text{ using} & : & EST(2) \\ 
\text{Adapt }\varepsilon \text{\ for }p\text{\ using} & : & 
EST_{c}:=||\nabla \cdot u_{n+1}||.%
\end{array}%
\end{equation*}%
The algorithm computes two velocity approximations. The first $u^{1}$ is
first-order and $A-$stable for all combinations of time-step and $%
\varepsilon $. The second $u$ is second-order, A-stable for constant (or
decreasing) time-step but only $0-$stable for increasing time-steps.
Variable (1 or 2) order is introduced as follows. The local error in each
approximation is estimated. If both are above the tolerance, the step is
repeated. Otherwise, the optimal next time-step is predicted for each method
by 
\begin{eqnarray*}
\text{first-order prediction}\text{: } &&k_{n+1}=k_{n}\left( \frac{TOL_m}{%
EST(1)}\right) ^{1/2}, \\
\text{second-order\ prediction}\text{: } &&k_{n+1}=k_{n}\left( \frac{TOL_m}{%
EST(2)}\right) ^{1/3}
\end{eqnarray*}%
The actual $k_{n+1}$ presented\ below and in the tests in Section 5 is
restricted to be $(0.5$ to $2.0)\times k_{n}$ and includes a safety factor
of $0.9$.

\begin{algorithm}[Adaptive order, $k$, $\protect\varepsilon $]
\texttt{Given }$TOL_m$\texttt{, }$TOL_{c}$\texttt{, } $u_{n}\mathtt{,\ }u_{n-1}
$, $u_{n-2}$, \texttt{previous second difference} $D_{2}(n)$ \texttt{and }$%
k_{n+1}$, $k_{n}$, $k_{n-1}$\texttt{\ }

\texttt{Compute: }$\tau =\frac{k_{n+1}}{k_{n}},$\ $\alpha_1=\frac{{\LARGE %
\tau (1.0+\tau )}}{{\LARGE 1.0+2.0\tau }},$\ $\alpha_2=\frac{\tau
_{n}(\tau _{n+1}\tau _{n}+\tau _{n}+1)(4\tau _{n+1}^{3}+5\tau
_{n+1}^{2}+\tau _{n+1})}{3(\tau _{n}\tau _{n+1}^{2}+4\tau _{n}\tau
_{n+1}+2\tau _{n+1}+\tau _{n}+1)}$

\texttt{Select: \ \ }$\hat{\varepsilon}=\sqrt{\varepsilon
_{n+1}\varepsilon _{n}}$ $\mathtt{or}$ $\hat{\varepsilon}=\min
\{\varepsilon _{n+1},\varepsilon _{n}\}.$\ 

\texttt{Set: \ \ }$u^{\ast }=\left( 1+\tau \right) u_{n}-\tau u_{n-1}.$

\texttt{Find BE approximation} $u_{n+1}^{1}$%
\begin{equation*}
\frac{u_{n+1}^{1}-u_{n}}{k_{n+1}}+u^{\ast }\cdot \nabla u_{n+1}^{1}+{\frac{1%
}{2}}(\nabla \cdot u^{\ast })u_{n+1}^{1}-\frac{k_{n+1}}{\varepsilon _{n+1}}%
\nabla \nabla \cdot u_{n+1}^{1}-\nu \Delta u_{n+1}^{1}=f_{n+1}-\frac{%
\hat{\varepsilon}}{\varepsilon _{n+1}}\nabla p_{n}.
\end{equation*}

\texttt{Compute difference }$D_{2}$\texttt{\ and updated velocity}%
\begin{eqnarray*}
D_{2}(n+1) &=&\frac{2k_{n}}{k_{n}+k_{n+1}}u_{n+1}^{1}-2u_{n}+\frac{2k_{n+1}}{%
k_{n}+k_{n+1}}u_{n-1} \\
u_{n+1} &=&u_{n+1}^{1}-\frac{\alpha_1}{2}D_{2}(n+1)
\end{eqnarray*}

\texttt{Compute estimators}%
\begin{eqnarray*}
EST(1) &=&\frac{\alpha_1}{2}\left\Vert D_{2}(n+1)\right\Vert , \\
EST(2) &=&\frac{\alpha_2}{6}\left\Vert \frac{3k_{n-1}}{%
k_{n+1}+k_{n}+k_{n-1}}D_{2}(n+1)-\frac{3k_{n-1}}{k_{n+1}+k_{n}+k_{n-1}}%
D_{2}(n)\right\Vert , \\
ESTc &=&||\nabla \cdot u_{n+1}||.
\end{eqnarray*}

\texttt{IF} $EST_{c}>TOL_{c}$ \ \texttt{or} $\min \{EST(1),EST(2)\}>TOL_{m}$ 
\texttt{THEN repeat step, resetting} $\varepsilon _{n+1},k_{n+1}$\ \texttt{by%
} 
\begin{eqnarray*}
\varepsilon _{n+1} &=&\max \{0.9\varepsilon _{n+1}\frac{TOL_c}{EST_{c}}%
,0.5\varepsilon _{n+1}\} \\
STEPBE &=&0.9\ast \left( \frac{TOL_m}{EST(1)}\right) ^{1/2}\ \ \ \max \left\{
0.9k_{n}\left( \frac{TOL_{m}}{EST(1)}\right) ^{1/2},0.5k_{n+1}\right\} \\
STEPFilter &=&0.9\ast \left( \frac{TOL_m}{EST(2)}\right) ^{1/3}\max \left\{
0.9k_{n}\left( \frac{TOL_{m}}{EST(2)}\right) ^{1/3}0.5k_{n+1}\right\} \\
k_{n+1} &=&\max \{STEPBE,STEPFilter\}
\end{eqnarray*}

\texttt{ELSE Predict }$\varepsilon ,k$\texttt{\ for each approximation:}%
\begin{eqnarray*}
STEPBE &=&\max \left\{ \min \left\{ 0.9k_{n+1}\left( \frac{TOL_{m}}{EST(1)}%
\right) ^{1/2},2k_{n+1}\right\} ,0.5k_{n+1}\right\} \\
STEPFilter &=&\max \left\{ \min \left\{ 0.9k_{n+1}\left( \frac{TOL_{m}}{%
EST(2)}\right) ^{1/3},2k_{n+1}\right\} ,0.5k_{n+1}\right\} \\
\varepsilon _{n+2} &=&\max \{\min \{0.9\varepsilon _{n+1}\frac{TOL_c}{EST_{c}}%
,2\varepsilon _{n+1}\},0.5\varepsilon _{n+1}\}
\end{eqnarray*}%
\ 

\texttt{Select method order with larger next step:}

\texttt{IF} ($STEPBE>STEPFilter$) \texttt{Then}

$\ \ \ \ \ \ \ \ u_{n+1}=$ \ $u_{n+1}^{1}$

\ \ \ \ \ \ \ $k_{n+2}=STEPBE$

\ \texttt{\ ELSE}\ \ \ \ \ \ $\ k_{n+2}=STEPFilter$

\ \ \texttt{ENDIF}

\texttt{Update pressure: \ }$p_{n+1}=\frac{\hat{\varepsilon}}{%
\varepsilon _{n+1}}p_{n}-\frac{k_{n+1}}{\varepsilon _{n+1}}\nabla \cdot
u_{n+1}.$

\texttt{Proceed to next step}
\end{algorithm}

The fixed order methods can, if desired, be implemented by commenting out
parts of the variable order Algorithm 4.3.

\section{Three Numerical Tests}

The stability and accuracy of the new methods are interrogated in two
numerical tests and the three discrete continuity equations are compared in
our third test. The tests employ the finite element method to discretize
space, with Taylor-Hood ($\mathbb{P}_{2}/\mathbb{P}_{1}$) elements, \cite%
{G89}. All the stability results proven herein hold for this spatial
discretization by essentially the same proofs. The meshes used for both
tests are generated using a Delaunay triangulation. The software package
FEniCS is used for both experiments \cite{Al15}.

We begin with comparative tests of the adaptive $k,\varepsilon $, first and
second-order method. Both adapt $\varepsilon $\ based on $||\nabla \cdot u||$%
. The first-order method accepts the first-order approximation $u_{n+1}^{1}$
and adapts the time-step based on $EST(1)$. The second-order method accepts $%
u_{n+1}$\ as the approximation and adapts the time step based on $EST(2)$.

\subsection{Test 1: Flow Between Offset Circles}

To interrogate stability and accuracy of the GA-method, we present the
results of two numerical tests. Pick%
\begin{gather*}
\Omega =\{(x,y):x^{2}+y^{2}\leq r_{1}^{2}\text{ and }%
(x-c_{1})^{2}+(y-c_{2})^{2}\geq r_{2}^{2}\}, \\
r_{1}=1,r_{2}=0.1,c=(c_{1},c_{2})=(\frac{1}{2},0), \\
f=\min \{t,1\}(-4y(1-x^{2}-y^{2}),4x(1-x^{2}-y^{2}))^{T},\text{ for }0\leq
t\leq 10.
\end{gather*}%
with no-slip boundary conditions on both circles and $\nu=0.001$. The finite element discretization has a maximal mesh width of $h_{max} = 0.0133,$ and the flow was solved using the direct solver UMFPACK \cite{D06}. For this test, we use fixed tolerances $TOL_m=TOL_c=0.001$. The flow (inspired by the
extensive work on variants of Couette flow, \cite{EP00}), driven by a
counterclockwise force (with $f\equiv 0$ at the outer circle), rotates about 
$(0,0)$ and interacts with the immersed circle. This induces a von K\'{a}rm%
\'{a}n vortex street which re-interacts with the immersed circle creating
more complex structures. There is also a central (polar) vortex that
alternately self-organizes then breaks down. Each of these events includes a
significant pressure response.

For both approximations we track the evolution of $k_{n}$ and $\varepsilon
_{n}$, the pressure at the origin, the violation of incompressibility, and
the algorithmic energy $\Vert u_{h}^{n+1}\Vert ^{2}+\varepsilon _{n+1}\Vert
p_{h}^{n+1}\Vert ^{2}$. These are all depicted in Figure \ref{fig=stab}
below. 
\begin{figure}[t]
	\subfloat[Timestep evolution]{\includegraphics[width=.49\textwidth]{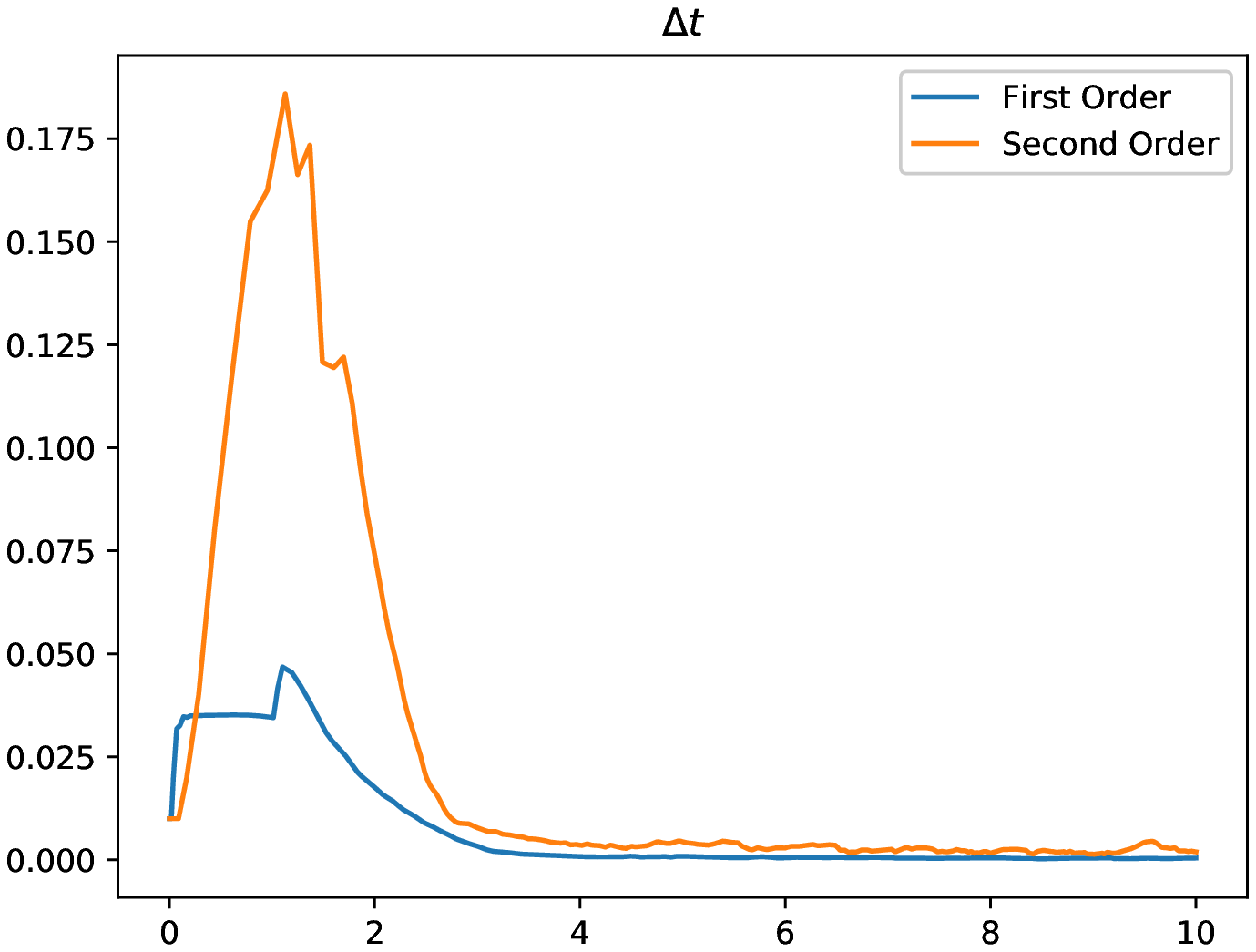}\label{fig1:a}}
    \subfloat[$\varepsilon$ evolution]{\includegraphics[width=.49\textwidth]{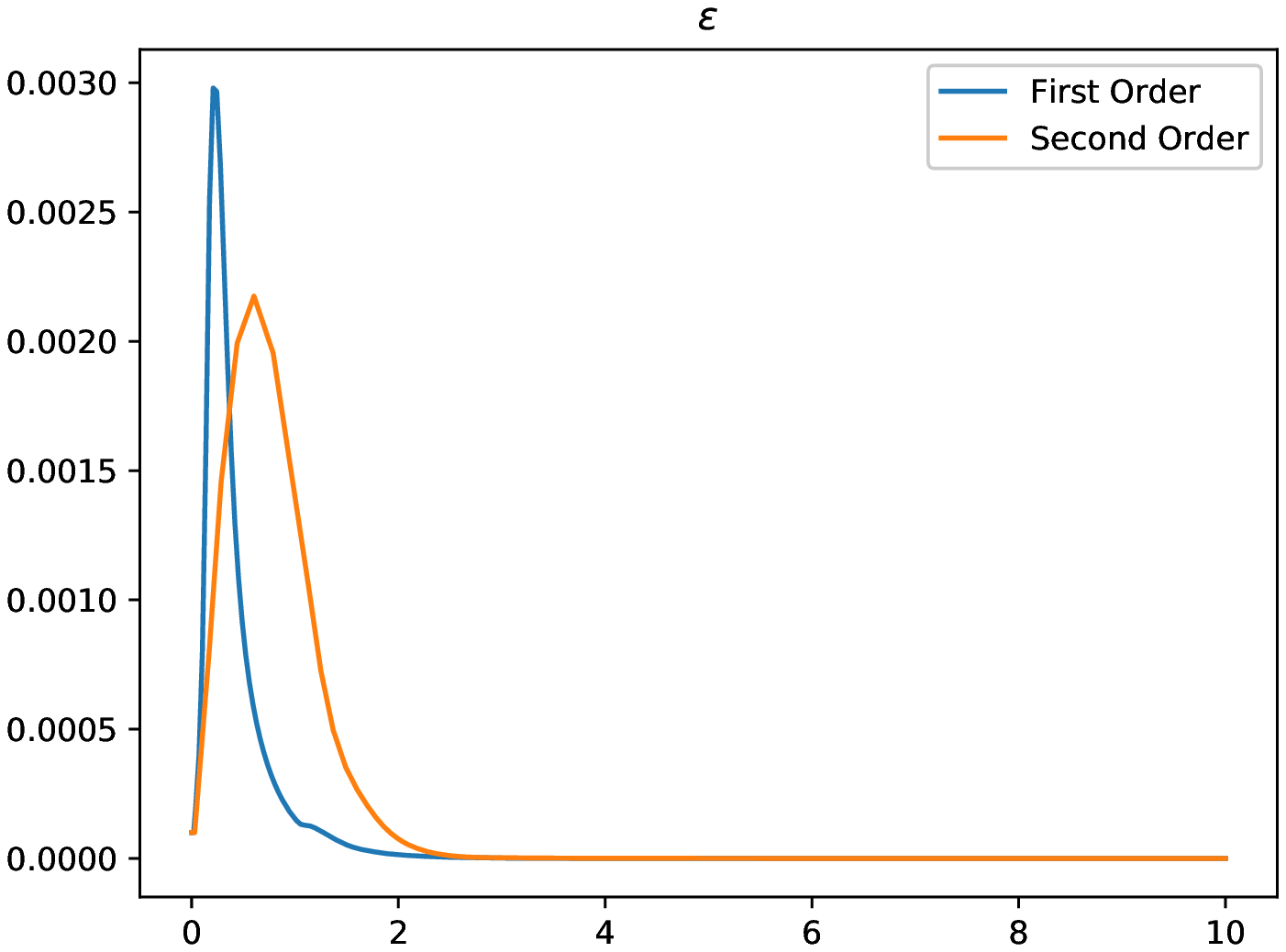}\label{fig1:b}}\\[2mm]
    \subfloat[Pressure at the origin]{\includegraphics[width=.49\textwidth]{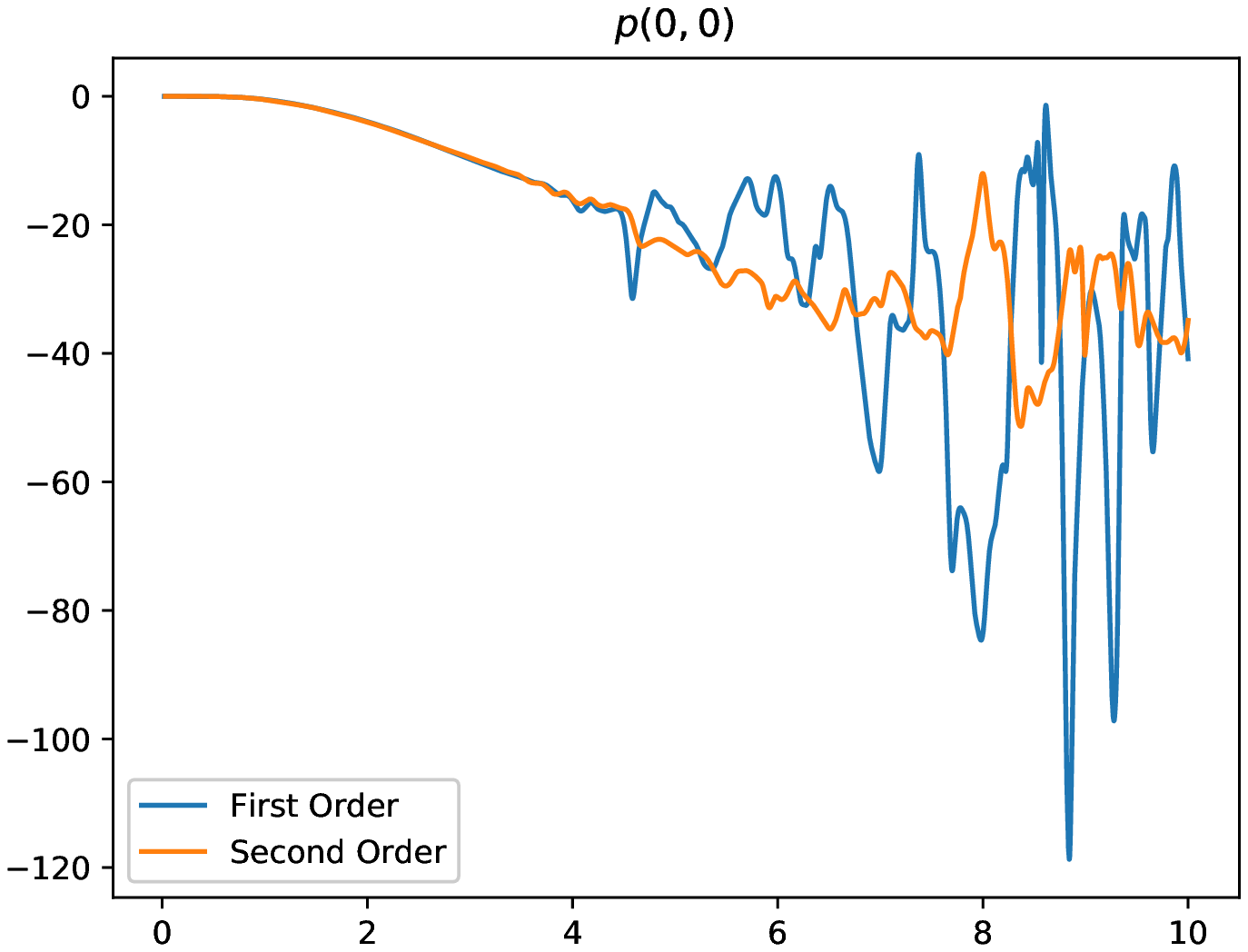}\label{fig1:c}}
    \subfloat[Divergence evolution]{\includegraphics[width=.49\textwidth]{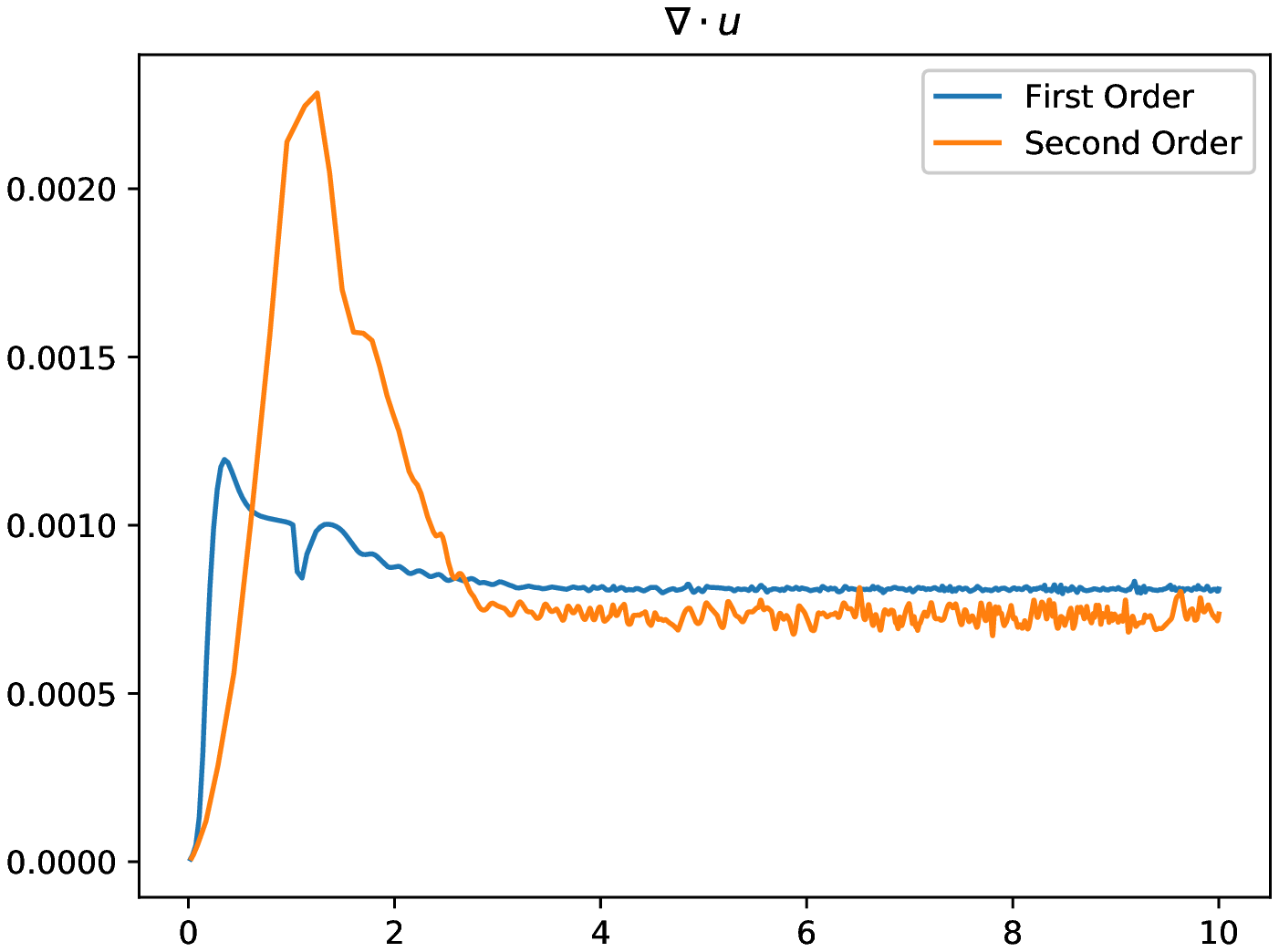}\label{fig1:d}}\\[2mm]
    \subfloat[Energy evolution]{\includegraphics[width=.49\textwidth]{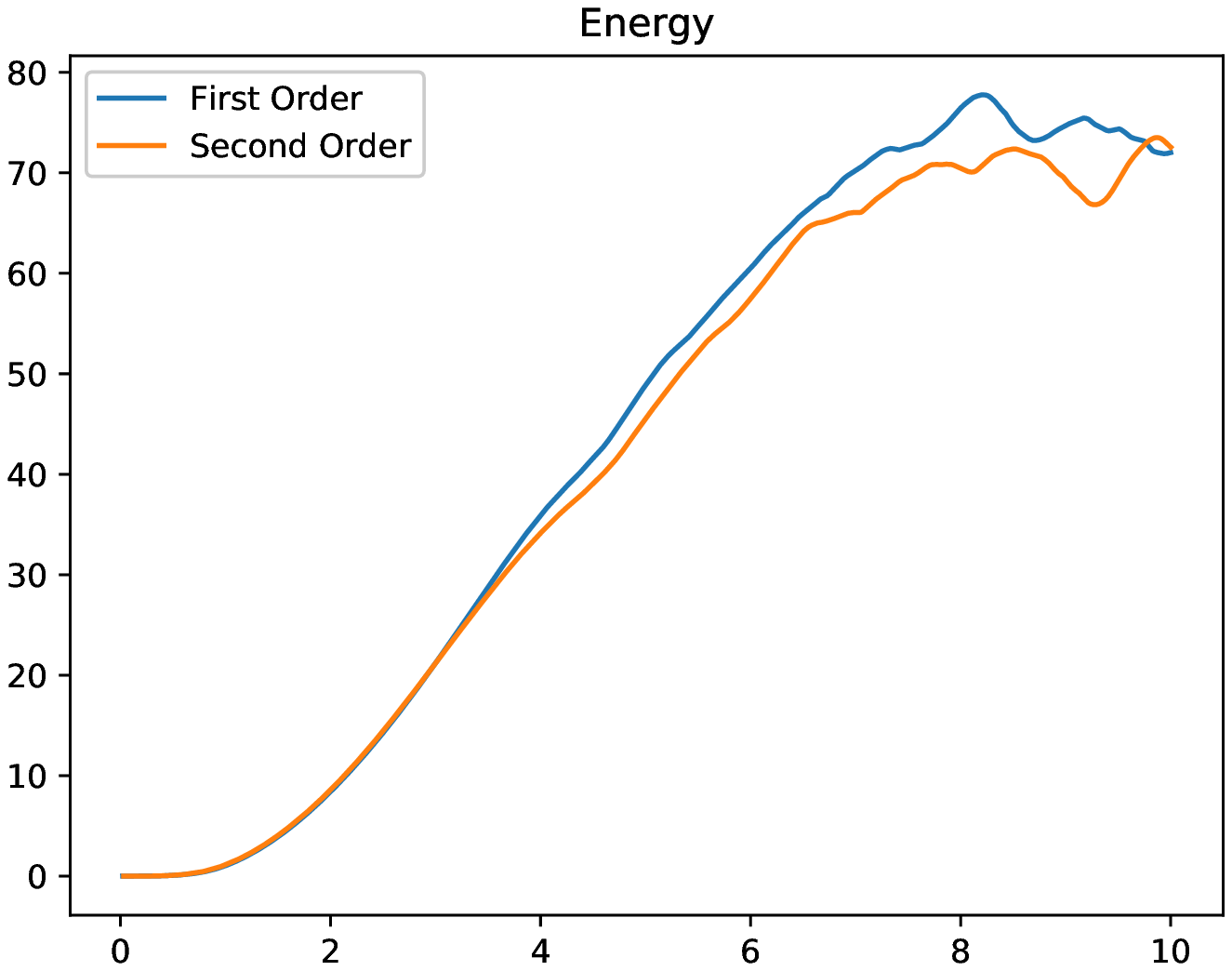}\label{fig1:e}}
	\caption{Stability and adaptability results.}
	\label{fig=stab}
	\end{figure}
Figure \ref{fig1:a} shows that the second-order scheme consistently
chooses larger time-steps than the first-order method. The evolution of $%
\varepsilon $, in Figure \ref{fig1:b}, behaves similarly for both
methods once the flow evolves. In testing AC methods pressure initialization
often causes irregular, transient spiky behavior near $t=0$ such as in
Figures \ref{fig1:a}, \ref{fig1:b}, \ref{fig1:d}.

The behavior of the pressure at the origin, $p(0,0;t)$ vs. $t$, is depicted
in Figure \ref{fig1:c}. To our knowledge, there is no convergence theory for AC
methods (or even fully coupled methods) which implies maximum norm
convergence for the pressure over significant time intervals and for larger
Reynolds numbers. Still, the irregular behavior observed in approximate
solutions, while not conforming to a convergence theory, reflects vortex
events across the whole domain and is interesting to compare. The profiles
of the pressure at the origin are similar for both methods over $0\leq t\leq
4$. For $t>4,$ $p(0,0;t)$ for the second-order scheme is less oscillatory.
This is surprising because the first-order scheme has more numerical
dissipation. The divergence evolution of the schemes also differ in the
initial transient of $||\nabla \cdot u(t)||$. After the initial transient,
the divergence behavior is similar. It is also possible that the difference
in $||\nabla \cdot u||$ transients is due to the strategy of $\varepsilon -$%
adaptation being sub-optimal. The model energy of both methods is largely
comparable. We note that the model energy depends on the choices of $%
\varepsilon $\ made. Thus model energy is not expected to coincide exactly.
Generally, Figures \ref{fig1:d}--\ref{fig1:e} behave similarly for both
algorithms.

\subsection{Test 2: Convergence and Adaptivity}

The second numerical test concerns the accuracy and adaptivity of the
GA-method. Let $\Omega =]0,1[^{2}$, with $\nu =1$. Consider the exact
solution (obtained from \cite{GMS06} and applied to the Navier-Stokes equations)
\begin{gather*}
u=\pi \sin {t}(\sin {2\pi y}\sin ^{2}{\pi x},-\sin {2\pi x}\sin ^{2}{\pi y})
\\
p=\cos {t}\cos {\pi x}\sin {\pi y},
\end{gather*}%
and consider a discretization of $\Omega$ obtained by 300 nodes on each
edge of the square. We proceed by running five experiments, adapting both
the first- and second-order schemes using the algorithms above, where the
tolerance for the continuity and momentum equations is $10^{-(.25i+3)}$ for $i=0,1,2,3,4$. To control the size of the timesteps, we require $k_n$ to be chosen such that $EST(1)\in (TOL_m/10,TOL_m)$. The solutions were obtained in parallel, utilizing the MUMPS direct solver \cite{A01}. To examine convergence, we
present in Figure \ref{fig=conv} log-log plots of the errors of the pressure
and the velocity against the average time-step taken during the test. We
also present semilog plots of the evolution of the pressure error and
timestep during the final test below. The plots show that the time-step adaptation
is working as expected and reducing the velocity error, Figure \ref{fig2:c}. Our intuition is that the pressure error is linked to satisfaction of incompressibility; however, Figure \ref{fig2:d} indicates convergence with respect to the timestep. In our calculations we did observe the
following: If $||\nabla \cdot u||$ is, e.g., two orders of magnitude
smaller then the tolerance, $\varepsilon $ is rapidly increased to be even $%
\mathcal{O}(1)$. At this point the pressure error and violation of
incompressibility spike upward and $\varepsilon $ is then cut rapidly. This
behavior suggests that a band of acceptable $\varepsilon $-values should be
imposed in the adaptive algorithm. 
\begin{figure}[t]
	\subfloat[Timestep and $\varepsilon$ evolution]{\includegraphics[width=.49\textwidth]{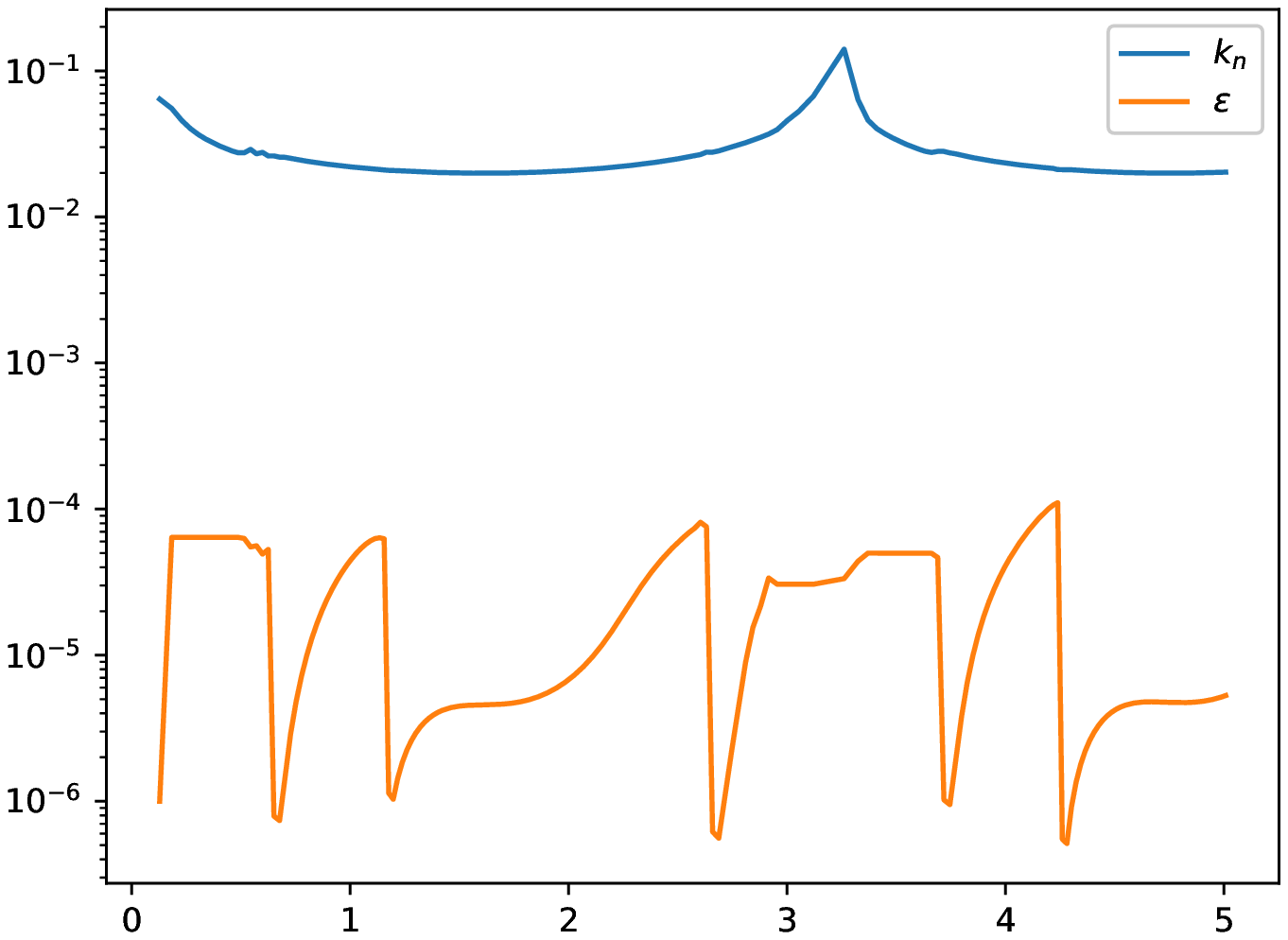}\label{fig2:a}}
    \subfloat[Pressure error evolution]{\includegraphics[width=.49\textwidth]{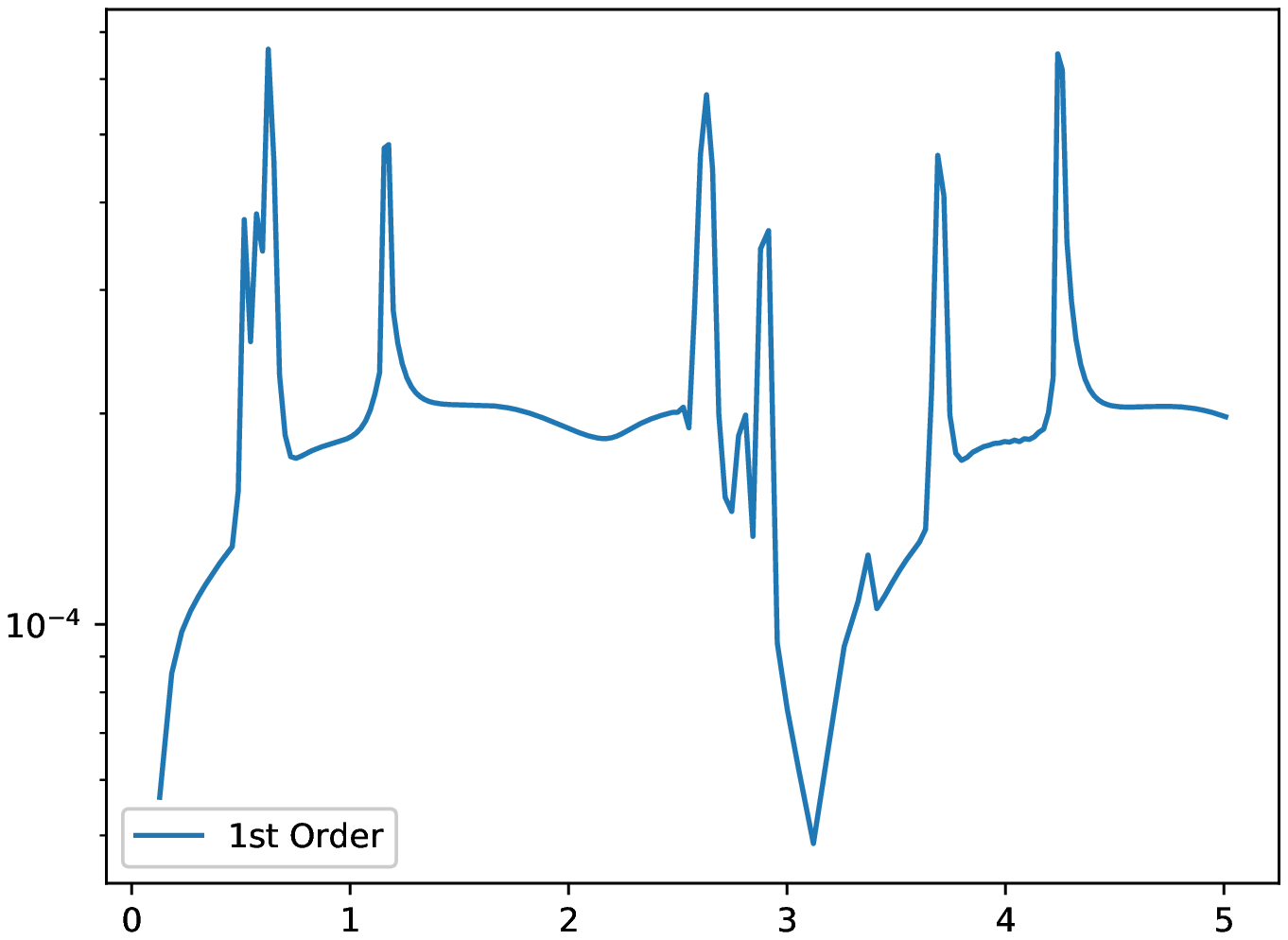}\label{fig2:b}}\\[2mm]
    \subfloat[Average timestep vs. velocity error]{\includegraphics[width=.49\textwidth]{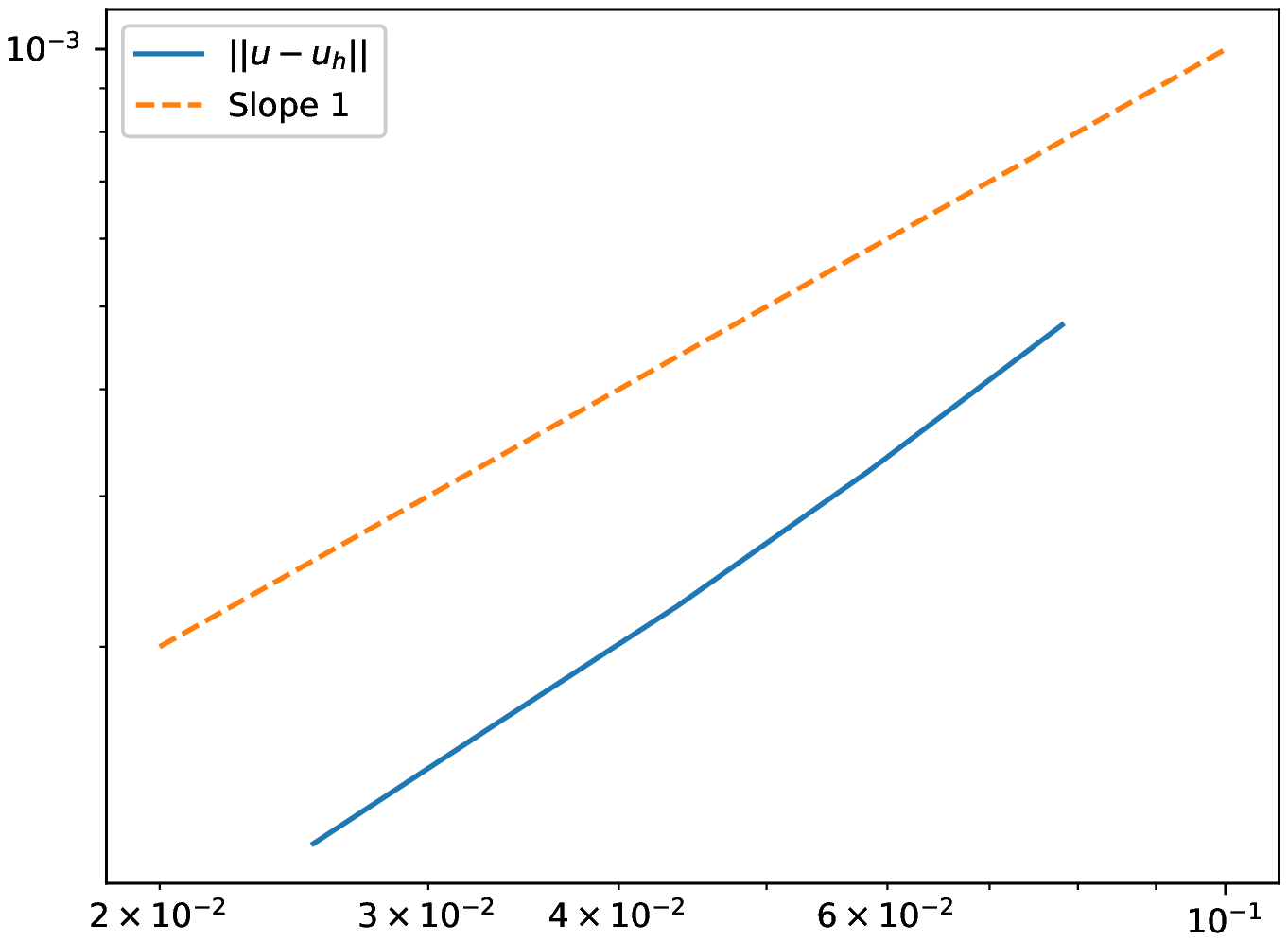}\label{fig2:c}}
    \subfloat[Average timestep vs. pressure error]{\includegraphics[width=.49\textwidth]{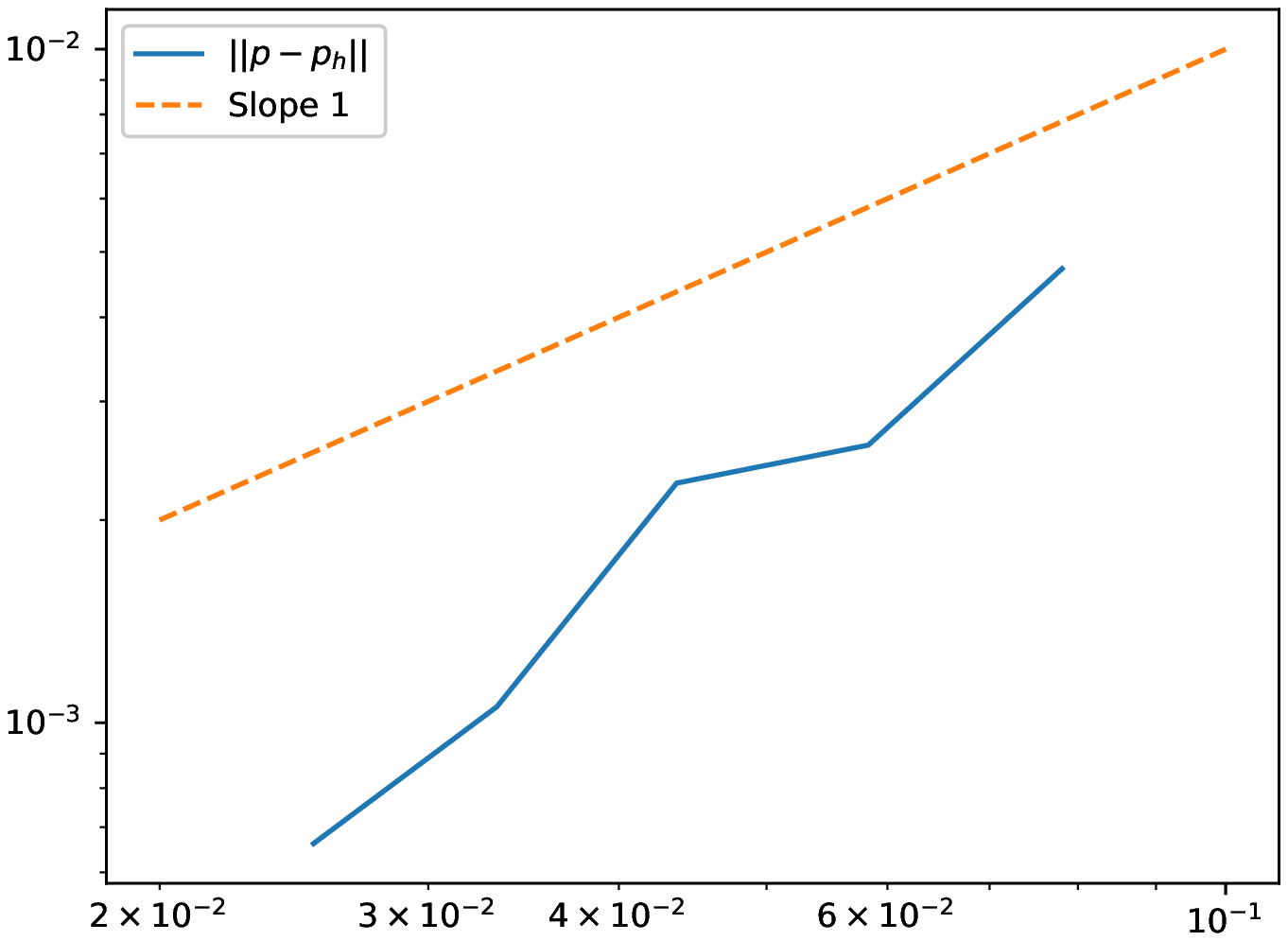}\label{fig2:d}}
	\caption{Accuracy and adaptability results.}
	\label{fig=conv}
	\end{figure}

To compare the GA, Min method and the scheme introduced in \cite{CLM18}, we
use the test problem given above in this section with a known exact
solution. The results are given in Figure \ref{fig=compare} below. Here, we use a
mesh with the same density and final time $T=1$. A timestep $k_{n}=10^{-2}$
is kept constant in this run to highlight differences in the evolution of
the variable $\varepsilon _{n}$, which has an initial value $\varepsilon
_{0}=10^{-4}$. These tests are preliminary: In them, the min-Method seems
preferable in error behavior but yields smaller values and thus less
well-conditioned systems. In the evolution of all four quantities, the GA-
and the CLM \cite{CLM18} method exhibit near identical behavior. The min-Method,
however, forces $\varepsilon $ to be an order of magnitude lower than the
values obtained by the other two schemes. This, in turn, forces the
divergence to be reduced. Furthermore, both the velocity and pressure errors
for the min-Method are smaller than those of the GA- and CLM-Methods.

\begin{figure}[t]
	\subfloat[$\varepsilon$ evolution]{\includegraphics[width=.49\textwidth]{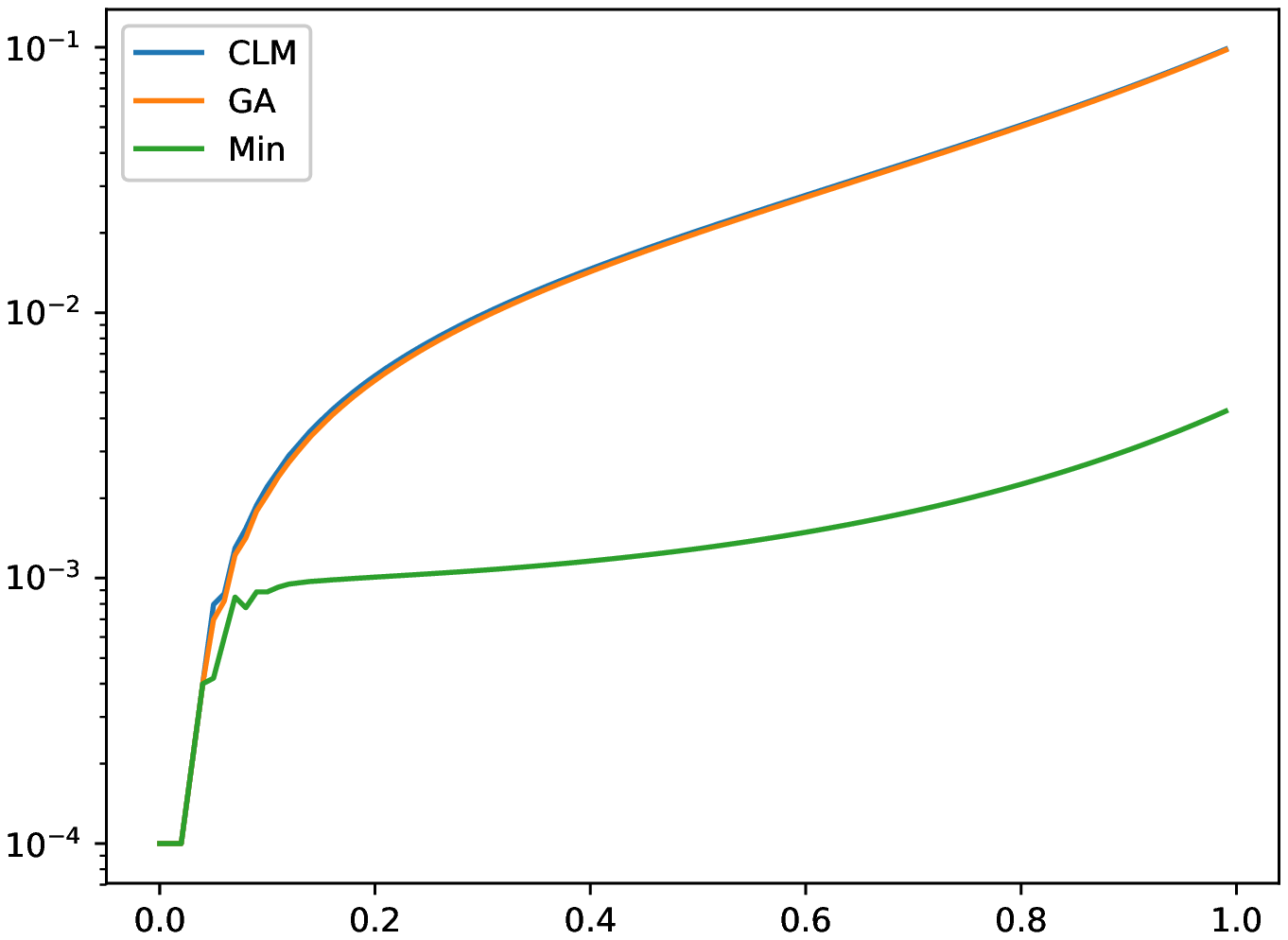}\label{fig3:a}}
    \subfloat[Divergence norm evolution]{\includegraphics[width=.49\textwidth]{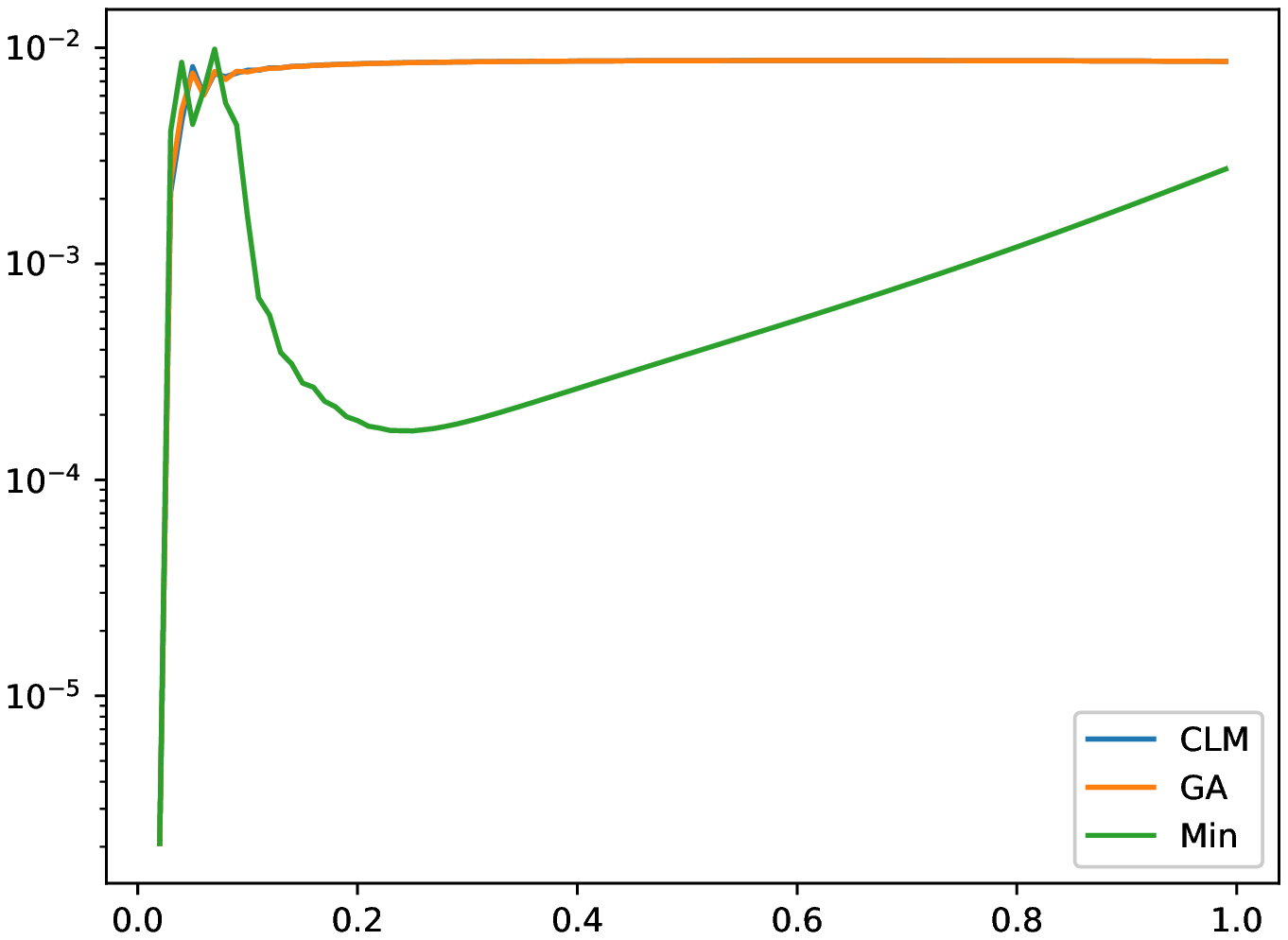}\label{fig3:b}}\\[2mm]
    \subfloat[Velocity error evolution]{\includegraphics[width=.49\textwidth]{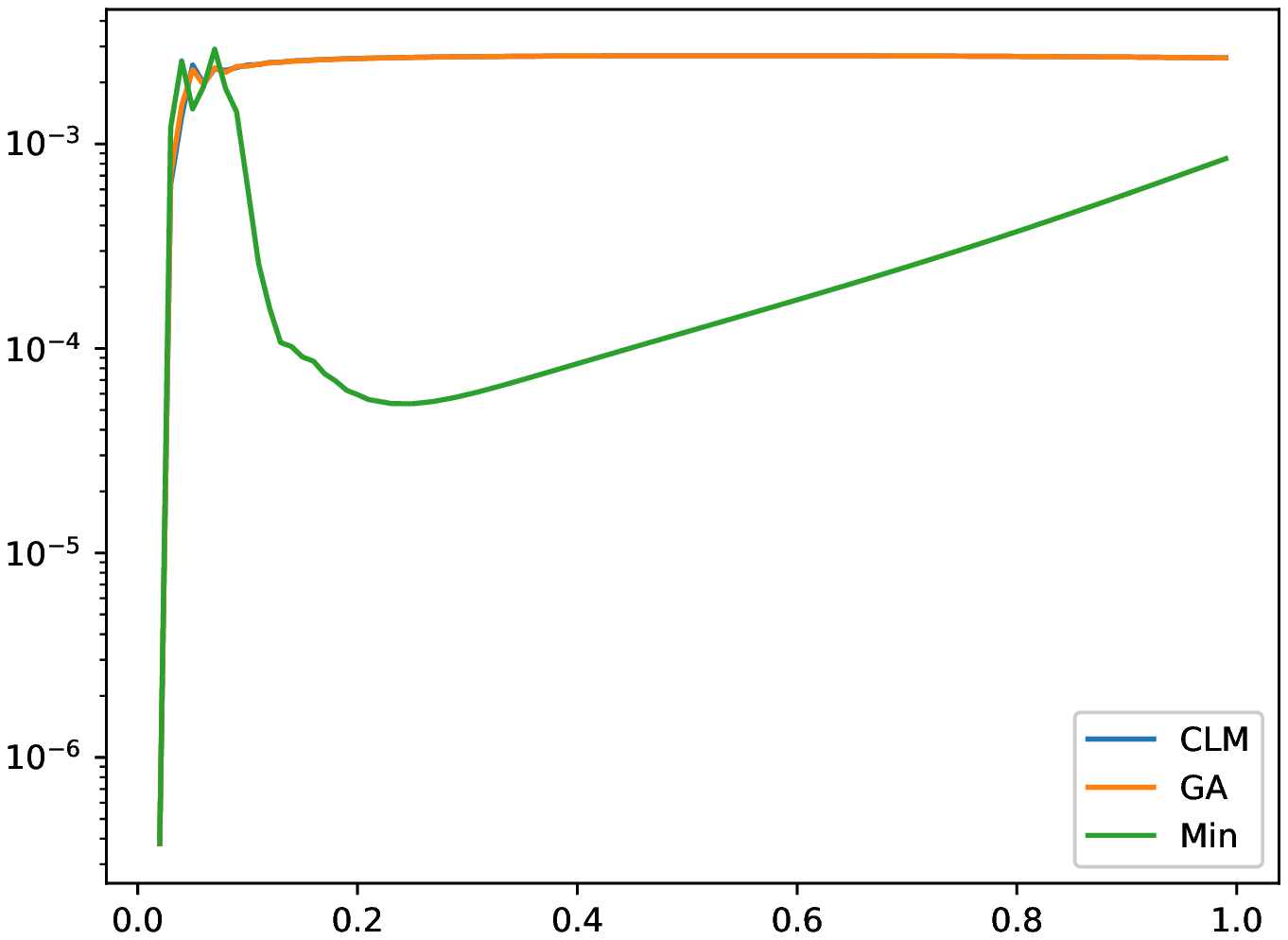}\label{fig3:c}}
    \subfloat[Pressure error evolution]{\includegraphics[width=.49\textwidth]{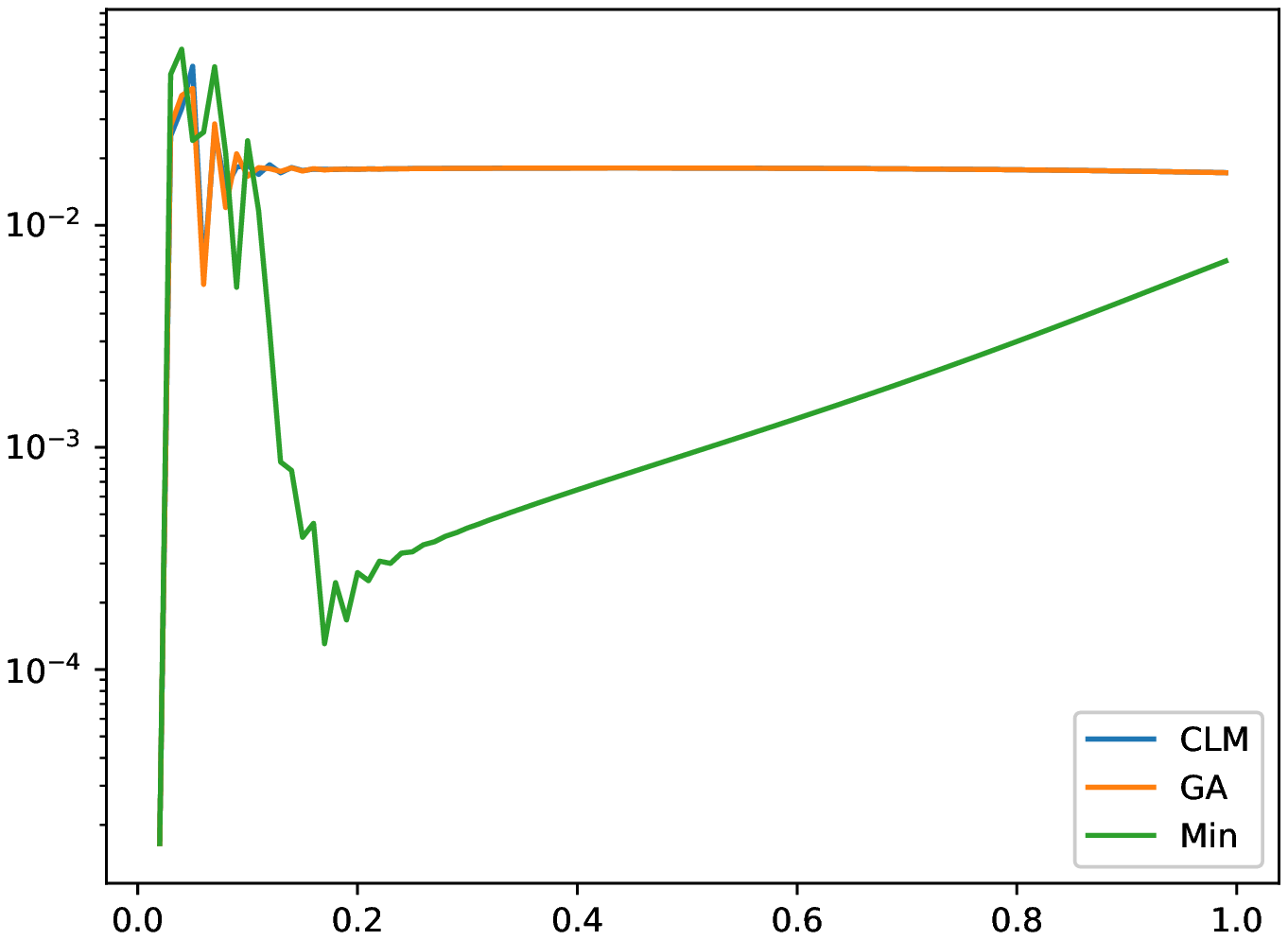}\label{fig3:d}}
	\caption{Comparison between GA, Min, and CLM methods.}
	\label{fig=compare}
	\end{figure}

\section{Conclusions, open problems and future prospects}

There are many open problems and algorithmic improvements possible. The
doubly adaptive algorithm selected smaller values of $\varepsilon $ than $k$
in our tests with the same tolerance for both. A further synthesis of the
methods herein with the modular grad-div algorithm of \cite{FLR18} would
eliminate any conditioning issues in the linear system arising. Developing
doubly adaptive methods of order greater than two (with modular grad-div) is
an important step to greater time accuracy. We mention in particular the new
embedded family of orders 2,3,4 of \cite{DLL18} as a natural extension. The
method of Dahlquist, Liniger and Nevanlinna \cite{DLN83} is unexplored for
PDEs, but has promise in CFD because it is A-stable for both increasing and
decreasing time-steps. Improved error estimators for the second-order method
herein would increase reliability. For AC methods, pressure initialization
and damping of nonphysical acoustics are important problems where further
progress would be useful.

\textbf{Open problems.} The idea of adapting independently $k$ and $%
\varepsilon $ is promising but new so there are many open problems. These
include:

\begin{itemize}
\item Is the $\varepsilon $-adaptation formula $\varepsilon
_{new}=\varepsilon _{old}(TOL/||\nabla \cdot u||)$\ improvable? Perhaps the
quotient should be to some fractional power. Perhaps adapting $\varepsilon $
should be based of a relative error in $||\nabla \cdot u||$, such as $%
||\nabla \cdot u||/||\nabla u||$. Analysis of the local (in time) error in $%
||\nabla \cdot u||$ is needed to support an improvement.
\end{itemize}

\begin{itemize}
\item The $\varepsilon $-adaptation\ strategy seems to need preset limits, $%
\varepsilon _{\min },\varepsilon _{\max }$, to enforce $\varepsilon _{\min
}\leq \varepsilon \leq \varepsilon _{\max }$. The preset of\ $\varepsilon
_{\min }$ is needed because $\nabla \cdot u=0$ cannot be enforced pointwise
in many finite element spaces. Finding a reasonable strategy for these
presets is an open problem. Similarly, it would be useful to develop a
coherent strategy for relating the two tolerances rather than simply picking
them to be equal (as herein).
\end{itemize}

\begin{itemize}
\item Proving convergence to a weak solution of the incompressible NSE of
solutions to the continuum analogs of the GA-method and min-Method for
variable $\varepsilon $ is an important open problem. In this analysis it is
generally assumed that $\varepsilon (t)\rightarrow 0$ in an arbitrary
fashion. A more interesting problem is to link $\varepsilon (t)$ and\ $%
||\nabla \cdot u||$ in the analysis. Similarly, an \'{a} priori error
analysis for variable $\varepsilon $\ is an open problem and may yield
insights on how the variance of $\varepsilon (t)$\ should be controlled
within an adaptive algorithm. The consistency error of the two methods are $%
\mathcal{O}(k+\varepsilon )$ and $\mathcal{O}(k^{2}+\varepsilon )$,
respectively. Energy stability has been proven herein for the first order
method and for the constant time-step, second order method. Thus, error
estimation while technical, should be achievable.
\end{itemize}

\begin{itemize}
\item Comprehensive testing of the variable (first or second) order method
is an open problem. VSVO methods are the most effective for systems of ODEs
but have little penetration in CFD. Testing the relative costs and accuracy
of VSVO in CFD is an important problem.
\end{itemize}


\begin{thebibliography}{9}
\bibitem{Al15} \textsc{M. Aln{\ae }s, J. Blechta, J. Hake, A.
Johansson, B. Kehlet, A. Logg, C. Richardson, J. Ring, M.E. Rognes, G.N.
Wells}, \emph{The FEniCS project version 1.5}, Archive of Numerical Software
3 (2015), 9--23.

\bibitem{A01} \textsc{P. Amestoy, I. Duff, J.-Y. L'Excellent, and J. Koster}, \emph{A fully asynchronous multifrontal solver using distributed dynamic scheduling}, SIAM Journal on Matrix Analysis and Applications 23 (2001), 15--41.

\bibitem{A72} \textsc{R.A. Asselin}, \emph{Frequency filter for time
integration}, Mon. Weather Review 100(1972), 487--490.

\bibitem{B76} \textsc{G.A. Baker}, \emph{Galerkin approximations for
the Navier-Stokes equations}, Technical Report,1976.



\bibitem{CHL12} \textsc{J. M. Connors, J. Howell, and W. Layton}, 
\emph{Decoupled time stepping for a fluid-fluid interaction problem}, SIAM
J. Numer. Anal. 50 (2012), pp. 1297-1319

\bibitem{CLM18} \textsc{R.M. Chen, W. Layton, and M. McLaughlin}, 
\emph{Analysis of variable step/non-autonomous artificial compression methods%
}, JMFM 21 (2018).

\bibitem{DLN83} \textsc{G. Dahlquist, W. Liniger and O. Nevanlinna}, 
\emph{Stability of two-step methods for variable integration steps} , SIAM
J. Numer. Anal. 20 (1983), 1071--1085.

\bibitem{D06} \textsc{T. Davis}, \emph{ Direct methods for sparse linear systems}, SIAM, vol. 2, 2006.

\bibitem{DLM17} \textsc{V. DeCaria, W. Layton and M. McLaughlin}, 
\emph{A conservative, second-order, unconditionally stable artificial
compression method}, CMAME 325 (2017), 733--747.

\bibitem{DLL18} \textsc{V. DeCaria, A. Guzel W. Layton and Yi Li, }%
\emph{A new embedded variable stepsize, variable order family of low
computational complexity,}\ https://arxiv.org/abs/1810.06670, 2018.

\bibitem{DLZ18} \textsc{V. DeCaria, W. Layton and Haiyun Zhao, }\emph{%
Analysis of a low complexity, time-accurate discretization of the
Navier-Stokes equations}, https://arxiv.org/abs/1810.06705, 2018.

\bibitem{EP00} \textsc{C. Egbers and G. Pfister, }\emph{Physics of
rotating fluids}, Springer LN in Physics 549 (2018).

\bibitem{FLR18} \textsc{J. Fiordilino, W. Layton, and Y. Rong}, \emph{%
An efficient and modular grad-div stabilization}, Computer Methods in
Applied Mechanics and Engineering 335 (2018), 327--346.


\bibitem{GMS06} \textsc{J.-L. Guermond, P. Minev and J. Shen}, \emph{An
overview of projection methods for incompressible flows}, Comput. Methods
Appl. Mech. Engrg. 195 (2006), 6011--6045.

\bibitem{GM15} \textsc{J.-L. Guermond and P. Minev}, \emph{High-Order
Time Stepping for the Incompressible Navier--Stokes Equations}, SIAM J. Sci.
Comput. 37-6 (2015), A2656-A2681 http://dx.doi.org/10.1137/140975231.

\bibitem{GM17} \textsc{J.-L. Guermond and P. Minev}, \emph{High-order
time stepping for the Navier--Stokes equations with minimal computational
complexity}, JCAM 310 (2017), 92--103.

\bibitem{GM18} \textsc{J.-L. Guermond and P. Minev}, \emph{High-order,
adaptive time stepping scheme for the incompressible Navier--Stokes equations%
}, technical report 2018.

\bibitem{G89} \textsc{M.D. Gunzburger}, \emph{Finite Element Methods
for Viscous Incompressible Flows - A Guide to Theory, Practices, and
Algorithms}, Academic Press, 1989.

\bibitem{GL18} \textsc{A. Guzel and W. Layton,} \emph{Time filters
increase accuracy of the fully implicit method}, BIT Numerical Mathematics,
58 (2018), 301-315.

\bibitem{HPG15} \textsc{A. Hay, S. Etienne, D. Pelletier and A. Garon}%
, \emph{hp-Adaptive time integration based on the BDF for viscous flows}.
JCP, 291 (2015), 151-176.

\bibitem{HJ07} \textsc{J. Hoffman and C. Johnson}, \emph{Computational
turbulent incompressible flow: Applied mathematics: Body and soul 4} (Vol.
4). Springer, Berlin, 2007.

\bibitem{JL04} \textsc{H. Johnston and J.-G. Liu}, \emph{Accurate,
stable and efficient Navier-Stokes solvers based on an explicit treatment of
the pressure term}, JCP 199(2004) 221-259.


\bibitem{KGGS10} \textsc{D.A. Kay, P.M. Gresho, P.M., Griffiths and
D.J. Silvester}, \emph{Adaptive time-stepping for incompressible flow Part
II: Navier--Stokes equations.} SIAM Journal on Scientific Computing,
32(2010), 111-128.

\bibitem{K02} \textsc{G.M. Kobel'kov,} \emph{Symmetric approximations
of the Navier-Stokes equations}, Sbornik: Mathematics. 193(2002), 1027-1047.

\bibitem{LLT16} \textsc{W. Layton, Y. Li, and C. Trenchea}, \emph{%
Recent developments in IMEX methods with time filters for systems of
evolution equations}, J. Comp. Applied Math. 299 (2016), 50--67.

\bibitem{OA10} \textsc{T. Ohwada and P. Asinari}, \emph{Artificial
compressibility method revisited: Asymptotic numerical method for
incompressible Navier Stokes equations}. J. Comp. Physics, 229:16981723,
2010.


\bibitem{P97} \textsc{A. Prohl}, \emph{Projection and
quasi-compressibility methods for solving the incompressible Navier-Stokes
equations}, Springer, Berlin, 1997.

\bibitem{R69} \textsc{A. Robert}, \emph{The integration of a spectral
model of the atmosphere by the implicit method}, Proc. WMO/IUGG Symposium on
NWP, Japan Meteorological Soc. , Tokyo, Japan, pp. 19-24, 1969.

\bibitem{S96} \textsc{J. Shen}, \emph{On a new pseudocompressibility
method for the incompressible Navier-Stokes equations}, Appl. Numer. Math.
21 (1996), 71--90.

\bibitem{S92a} \textsc{J. Shen}, \emph{On error estimates of
projection methods for the Navier-Stokes equations: First-Order Schemes},
SINUM 29(1992) 57-77.

\bibitem{S92} \textsc{J. Shen}, \emph{On error estimates of higher
order projection and penalty-projection schemes for the Navier-Stokes
equations}, Numer. Math. 62(1992) 49-73.

\bibitem{S96a} \textsc{J. Shen}, \emph{On error estimates of the
projection method for the Navier-Stokes equations: second-order schemes,}%
Math. Comp. 65(1996)1039-1065.



\bibitem{TS87} \textsc{C. Temperton and A. Staniforth}, \emph{An
efficient two-time level semi-Lagrangian semi-implicit scheme}, Q.J.Royal
Meteor. Soc. 113(1987),1027-1039.

\bibitem{VV13} \textsc{A. Veneziani and U. Villa}, \emph{ALADINS: An
algebraic splitting time-adaptive solver for the incompressible
Navier-Stokes equations}, JCP 238(2013) 359-375.

\bibitem{W09} \textsc{P.D. Williams}, \emph{A proposed modification to
the Robert-Asselin time filter}, Monthly Weather Review, 137(2009),
2538-2546.

\bibitem{W11} \textsc{P.D. Williams}, \emph{The RAW Filter: An
Improvement to the Robert--Asselin Filter in Semi-Implicit Integrations},
Mon. Weather Rev., 139 (2011), 1996--2007.

\bibitem{YBC16} \textsc{L. Yang, S. Badia and R. Codina},\emph{\ A
pseudo-compressible variational multiscale solver for turbulent
incompressible flows}, Comp. Mechanics 58(2016) 1051-1069.

\bibitem{Z06} \textsc{R.Kh. Zeytounian}, \emph{Topics in hyposonic flow
theory}, Lecture Notes in Physics, Springer, Berlin, 2006.
\end{thebibliography}
\end{document}